\documentclass[a4paper,11pt]{article}
\author{Dane Rogers$^*$ and Matthias Winkel\thanks{University of Oxford, Department of Statistics, 24--29 St Giles', Oxford OX1 3LB, UK}}
\usepackage{amsmath}
\usepackage{amssymb}
\usepackage{enumerate}
\usepackage{latexsym}
\usepackage{graphicx}
\usepackage{amsthm}
\usepackage{color}
\usepackage{hyperref}

\numberwithin{equation}{section}

\theoremstyle{plain}
\newtheorem{dummy}{***}[section]
\newtheorem{thm}[dummy]{Theorem}

\newtheorem{lem}[dummy]{Lemma}
\newtheorem{cor}[dummy]{Corollary}

\newtheorem{prop}[dummy]{Proposition}

\newtheorem{conj}[dummy]{Conjecture}

\theoremstyle{definition}
\newtheorem{defn}[dummy]{Definition}

\theoremstyle{remark}

\newtheorem{rem}[dummy]{Remark}

\usepackage{graphicx}

\newcommand{\RR}{\mathbb{R}} 

\newcommand{\NN}{\mathbb{N}}

\newcommand{\DD}{\mathbb{D}}

\newcommand{\PP}{\mathbb{P}}
\newcommand{\LL}{\mathbb{L}}

\newcommand{\EE}{\mathbb{E}}

\newcommand{\Cc}{\mathcal{C}}
\newcommand{\Ll}{\mathcal{L}}

\newcommand{\Ss}{\mathcal{S}}

\newcommand{\Ff}{\mathcal{F}} 
 
\newcommand{\nl}{\varnothing} 

\title{\vspace{-2.2cm}$\;$\\ \bf A RAY--KNIGHT REPRESENTATION\\ OF UP-DOWN CHINESE RESTAURANTS}

\begin{document}

\maketitle

\vspace{-0.6cm}

\begin{abstract}
We study composition-valued continuous-time Markov chains that appear naturally in the framework of Chinese Restaurant
Processes (CRPs). 
As time evolves, new customers arrive (up-step) and existing customers leave (down-step) at suitable rates derived from the ordered CRP of Pitman and Winkel (2009).
We relate such up-down CRPs to the splitting trees of Lambert (2010) inducing spectrally positive
L\'evy processes. Conversely, we develop theorems of Ray--Knight type to recover more general up-down CRPs
from the heights of L\'evy processes with jumps marked by integer-valued paths.
We further establish limit theorems for the L\'evy process and the integer-valued paths to connect to work by Forman et
al. (2018+) on interval partition diffusions and hence to some long-standing conjectures.


\end{abstract}

\section{Introduction} 

The purpose of this paper is to study a class of continuous-time Markov chains in the state space\vspace{-0.2cm}
$$\mathcal{C}:=\big\{(n_1,\ldots,n_k)\colon k\ge 0,\ n_1,\ldots,n_k\ge 1\big\}$$
of integer compositions, which includes, for $k=0$, the empty vector that we also denote by $\varnothing$. Such Markov chains arise naturally in the framework of the 
Dubins--Pitman two-parameter Chinese Restaurant Process (CRP) \cite{pitman06}, when considered with the additional order structure of Pitman and Winkel \cite{pitmanwinkel09}. 
Specifically, we interpret $n_1,\ldots,n_k$ as the numbers of customers at an ordered list of $k$ tables in a restaurant, short \em table sizes\em. In the most basic model, 
we allow only the following transitions from $(n_1,\ldots,n_k)\in\mathcal{C}$.
\begin{itemize}
  \item At rate $n_i-\alpha$, a new customer joins the $i$-th table, leading to a transition into state $(n_1,\ldots,n_{i-1},n_i+1,n_{i+1},\ldots,n_k)$, $1\le i\le k$.
  \item At rate $\alpha$ a new customer opens a new table inserted directly to the right of the $i$-th table, leading to state $(n_1,\ldots,n_i,1,n_{i+1},\ldots,n_k)$, 
    $1\le i\le k$.
  \item At rate $\theta$ a new customer opens a new table inserted in the left-most position, leading to state $(1,n_1,\ldots,n_k)$.
  \item At rate 1 each existing customer at table $i$, $1\le i\le k$, leaves, leading to either $(n_1,\ldots,n_{i-1},n_i\!-\!1,n_{i+1},\ldots,n_k)$ if $n_i\!\ge\! 2$ or
    $(n_1,\ldots,n_{i-1},n_{i+1},\ldots,n_k)$ if $n_i\!=\!1$.  
\end{itemize}
These transition rates give rise to a $\mathcal{C}$-valued continuous-time Markov chain if $0\le\alpha\le 1$ and $\theta\ge 0$, which we call a \em continuous-time up-down 
ordered Chinese Restaurant Process with parameters $(\alpha,\theta)$\em, or, for the purposes of this paper, just \em up-down oCRP$(\alpha,\theta)$\em. 
\pagebreak

The first three bullet points increase the number of customers and we call any one of them an \em up-step\em, 
while the last bullet point decreases their number, a \em down-step\em.  
Conditionally given an up-step, the (induced discrete-time) transition probabilities are those of an ordered CRP \cite{pitmanwinkel09}. 
Without the order of tables, the up-step rates have been related to the usual Dubins--Pitman CRP \cite[Section 3.4]{pitman06},  
where the middle two bullet points combine to a rate $k\alpha+\theta$ for a new table; see also \cite{joycetavare87}. 
See \cite{james06,gnedinpitman05,pitmanyakubovich18} for other discussions of order structures related to CRPs.  
It was shown in \cite[Proposition 6]{pitmanwinkel09} that starting from $(1)\in\mathcal{C}$, the distribution $p_{n+1}$ after $n$ consecutive up-steps is the same as in 
(the left-right reversal) of a Gnedin--Pitman \cite{gnedinpitman05} regenerative composition structure that is known to be weakly sampling consistent, 
i.e.\ the push-forward of $p_{n+1}$ under a down-step is also $p_{n}$. Specifically, writing $N_i=n_i+\cdots+n_k$,\vspace{-0.1cm}
$$p_n(n_k,\ldots,n_1)=\prod_{i=1}^kr(N_i,n_i),\quad\mbox{where }r(n,m)= \binom{n}{m} \frac{(n\!-\!m)\alpha + m\theta}{n}\frac{(1\!-\!\alpha)_{(m-1)\uparrow}}{(n\!-\!m\!+\!\theta)_{m\uparrow}},\vspace{-0.1cm}$$
for $1 \leq m \leq n$, with $x_{y\uparrow}:=x(x+1)\cdots(x+y-1)$. This suggests to consider a \em discrete-time up-down oCRP$(\alpha,\theta)$ \em on 
$\mathcal{C}_n:=\big\{(n_1,\ldots,n_k)\in\mathcal{C}\colon n_1+\cdots+n_k=n\big\}$, $n\ge 1$, in which each transition consists of an up-step followed by a down-step. 
Then $p_n$ is stationary. Similar up-down chains on related state spaces were studied in \cite{borodinolshanski09,fulman09}, 
and for the corresponding discrete-time up-down CRP without the order of tables, Petrov \cite{petrov09} noted the stationary distribution,
which in our setting is obtained as the push-forward of $p_n$ under the map that ranks a vector into decreasing order. 
Petrov's main result was the existence of a diffusive scaling limit of his up-down chain, when represented in the infinite-dimensional simplex of decreasing sequences with sum 1.

\begin{conj} Let $(C^n_j)_{j\ge 0}$ be a $\mathcal{C}_n$-valued discrete-time up-down oCRP$(\alpha,\theta)$, for each $n\ge 1$. If $n^{-1}C_0^n$ converges, as $n\rightarrow\infty$, then $(n^{-1}C^n_{[n^2y]})_{y\ge 0}$ has a 
  diffusive scaling limit, when suitably represented on a space of interval partitions. For $0<\theta=\alpha<1$ and $0=\theta<\alpha<1$, these are the $(\alpha,\alpha)$- and $(\alpha,0)$-interval-partition evolutions of \cite{PartB}.
\end{conj}

While Petrov \cite{petrov09} used analytic methods studying how generators act on a certain core of symmetric functions, we develop here probabilistic methods in order to study
scaling limits associated with the (continuous-time) up-down oCRP$(\alpha,\theta)$. As demonstrated e.g.\ by Pal \cite{pal13} in a finite-dimensional setting, asymptotic results for
continuous-time Markov chains associated with discrete-time Markov chains via a method that can be referred to as ``Poissonization'' may sometimes be used to deduce asymptotic results for the 
discrete-time Markov chains via ``de-Poissonization''. See also Shiga \cite{Shiga1990}. Supporting the conjecture, it is confirmed in \cite[Theorem 3.1.2]{Rogers2020} that the mixing times of $(C^n_j)_{j\ge 0}$, in the sense of controlling the maximal separation distance of Aldous and Diaconis \cite{AldousDiaconis1987}, are of order $n^2$, extending results of Fulman \cite{fulman09} in the unordered case.

In the present paper, we do not explore further the passage between discrete and continuous time. We relate the 
(continuous-time) up-down oCRP$(\alpha,\theta)$ to genealogical trees \cite{geigerkersting97} and their jumping chronological contour processes (JCCPs) \cite{lambert10} that lead 
to representations as spectrally positive L\'evy processes, whose jumps we further mark by integer-valued paths. We show that the up-down 
oCRP$(\alpha,\theta)$, and natural generalisations, can be recovered from the heights of such a marked L\'evy process. 
This result is reminiscent of the recovery of a geometric Galton--Watson process from the occupation measure (upcrossing counts) of a suitably stopped simple symmetric random 
walk \cite{knight63} or indeed, in the scaling limit, the recovery of a Feller diffusion (squared Bessel process of dimension 0) as local time process of a stopped Brownian 
motion. 

In our framework, we establish scaling limits of the L\'evy process and of the integer-valued paths marking its jumps hence connecting to the framework of Forman et al. 
\cite{fourauthor16,PartA,PartB} and thereby to two long-standing conjectures. Specifically, Aldous \cite{AldousDiffusionProblem} conjectured the existence of a scaling limit for a simple up-down Markov chain with uniform stationary distribution on certain discrete binary trees as a diffusive evolution of a Brownian Continuum Random Tree \cite{AldousCRT1}. Feng and Sun \cite{FengSun2010} conjectured the existence of certain measure-valued diffusions whose stationary distributions are measures with Poisson--Dirichlet distributed atom sizes, also known as Pitman--Yor processes \cite{Teh2006} or three-parameter Dirichlet processes \cite{Carlton2002} in the Bayesian non-parametrics literature.

\subsection{Main results}

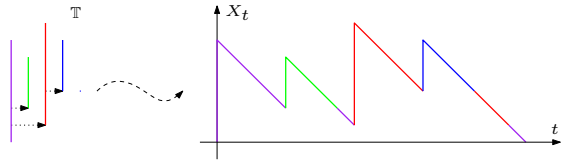
\begin{figure}[tbp] \centering \includegraphics[trim = 0mm 5mm 0mm 10mm, clip, scale=0.4]{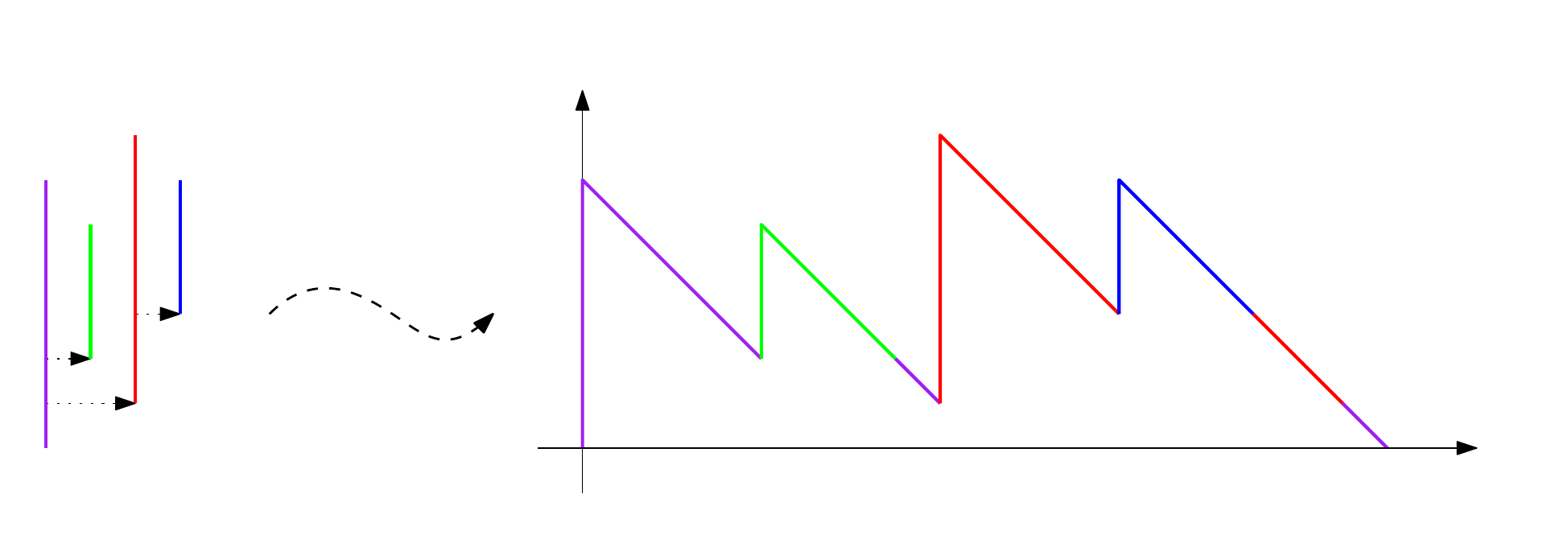}
\begin{picture}(0,0)
  \put(-142,58){$_{_{X_t}}$}
  \put(-20,13){$_{_t}$}
  \put(-200,57){$_{_{\mathbb{T}}}$}
\end{picture}
\vspace{-0.3cm}
\caption{The relationship between a genealogical tree $\mathbb{T}$ and its JCCP $X$.}
\label{Figure6}
\end{figure}

Let us first explain how an up-down oCRP$(\alpha,0)$, which we will denote by $(\mathcal{S}^y)_{y\ge 0}$, induces a genealogy. The rates stated at the beginning of the
introduction are such that every table evolves in size as an integer-valued Markov chain with up-rate $m-\alpha$ and down-rate $m$ when in state $m\ge 1$, and is removed 
when hitting state 0. Hence, tables have ``death'' times. Furthermore, while a table is ``alive'', new tables are inserted (``children born'') directly to its right at rate 
$\alpha$. Note the recursive nature of the model and the independence of table size evolutions. 
This is the genealogy \cite{geigerkersting97,lambert10} of a binary homogeneous Crump--Mode--Jagers (CMJ) branching process. Figure \ref{Figure6} captures this as a tree 
$\mathbb{T}$ with a vertical line for each table and horizontal arrows linking each parent table to its children at heights/levels corresponding to birth times.

Travelling around this genealogical tree $\mathbb{T}$ recording heights as in Figure \ref{Figure6}, yields Lambert's \cite{lambert10} jumping chronological contour process 
(JCCP) $X$: each vertical line yields a jump from a birth level to a death level, and exploration is by sliding down at unit speed and recursively jumping up at the birth levels of 
children. As jump heights are IID table lifetimes and jumps occur at the rate $\alpha$ of table insertions, this process is a L\'evy process starting from an initial 
table lifetime and stopped when reaching 0. We further mark each jump $U_i$ of the JCCP of height $\zeta_i\!=\!X_{U_i}\!-\!X_{U_i-}$ by the table size evolution 
$\big(Z^i(s)\big)_{0\le s\le\zeta_i}$, as in Figure \ref{Figure1}. In a general framework of CMJ 
processes, this is what Jagers \cite{jagers69,jagers75} studied: each individual has ``characteristics'' that vary during its lifetime. Key quantities of interest in a CMJ 
process are the characteristics $Z^i(y-X_{U_i-})$ at each time $y$, or summary statistics such as sums of these characteristics. 

Conversely, in a marked JCCP setting we can at each level $y$ extract a composition $\mathcal{S}^y$ by listing from left to right for each jump crossing level $y$ the 
size given by the mark for that level. In this construction, we refer to $(\mathcal{S}^y)_{y\ge 0}$ as the \em skewer process \em as we imagine piercing the marked JCCP at level 
$y$ and pushing together all sizes that we find at this level to form a sequence without gaps. See Figure \ref{Figure1}. This is the discrete analogue of the skewer process of 
\cite{PartA}. We provide a rigorous set-up in Section \ref{sec:skewer} and prove carefully the following result, formulated here in the setting of size evolutions on 
$\mathbb{N}_0:=\{0,1,2,\ldots\}$, which we call $Q_\alpha$-Markov chains, with Q-matrix $Q_\alpha$ whose only non-zero off-diagonal entries are $q_\alpha(m,m\!+\!1)=m\!-\!\alpha$ and
$q_\alpha(m,m\!-\!1)=m$ for $m\!\ge\! 1$.

\begin{figure} \centering \includegraphics[trim = 0mm 0mm 0mm 0mm, clip, scale=0.6]{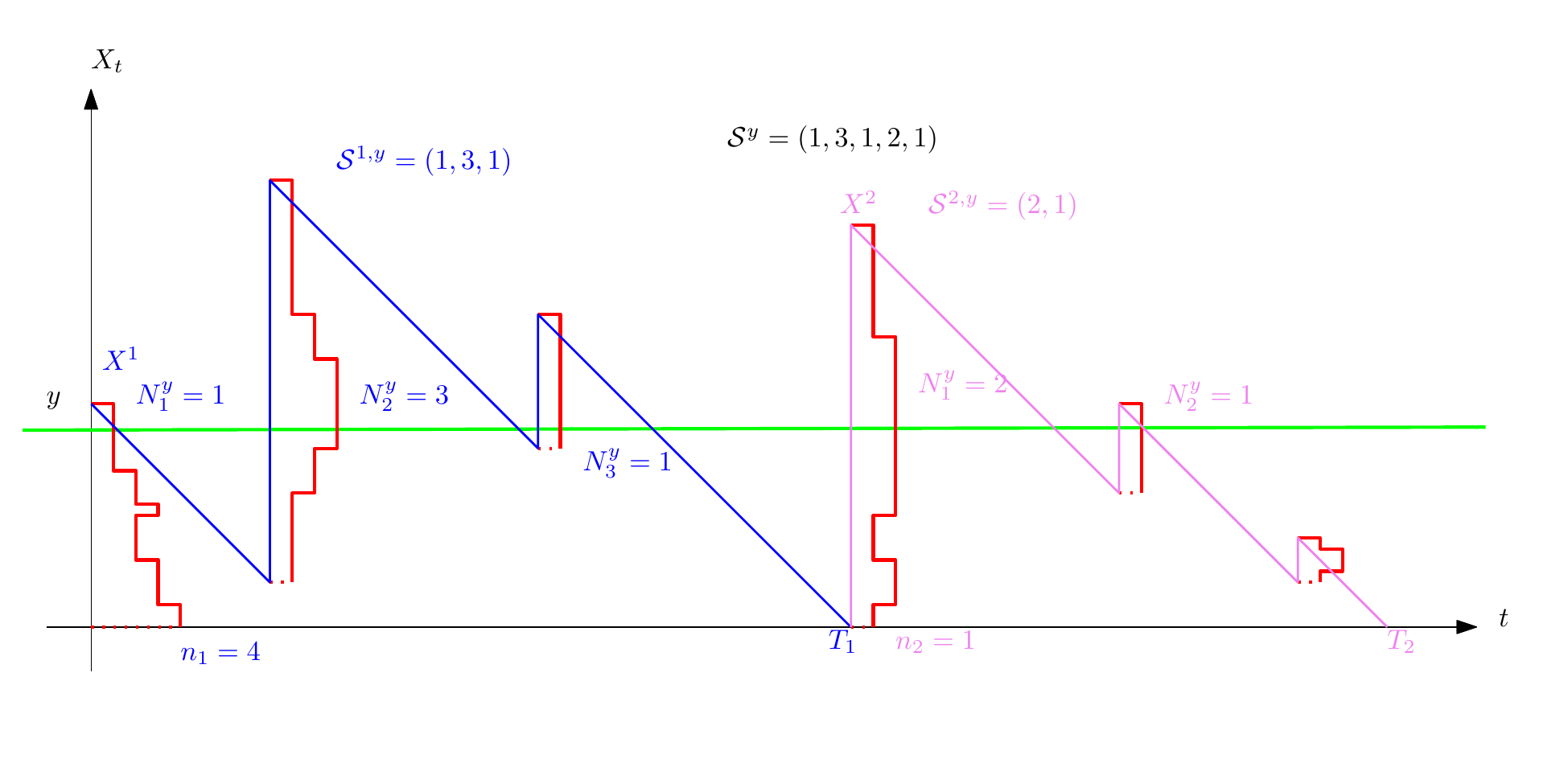}
\begin{picture}(0,0)
  \put(-318,26){\color{red}$^{^{Z^0}}$}
  \put(-275,37){\color{red}$^{^{Z^1}}$}
  \put(-228,50){\color{red}$^{^{Z^2}}$}
\end{picture}
\vspace{-0.8cm}
\caption[yt]{The skewer at level $y$ of the marked JCCP when $\theta=0$.\vspace{-0.2cm}}
\label{Figure1}
\end{figure}
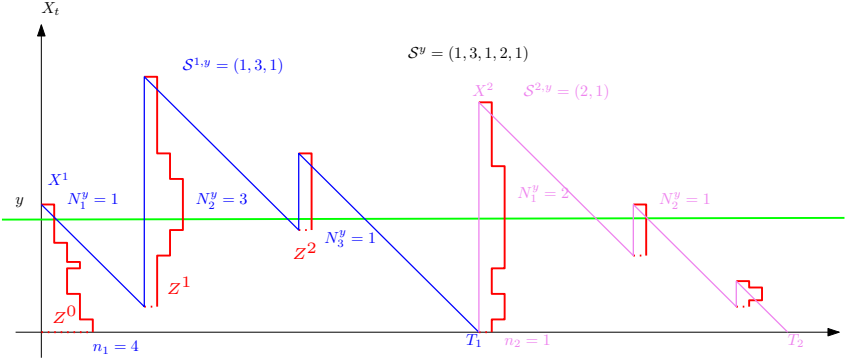

\begin{thm}\label{thm:skewer} Let $0\le\alpha\le 1$ and $n_0\ge 1$. Let $Z^i$ be independent $Q_\alpha$-Markov chains with $Z^0(0)=n_0$ and $Z^i(0)=1$ for $i\ge 1$. 
  Let $\zeta_i=\inf\{s\ge 0\colon Z^i(s)=0\}$ be the absorption times. Let $(J_t)_{t\ge 0}$ be an independent Poisson process of rate $\alpha$ and \vspace{-0.2cm}
  \begin{equation}\label{eqn:Levyprocess} X_t=-t+\sum_{i=0}^{J_t}\zeta_i,\quad t\ge 0,\qquad\mbox{and}\qquad T=\inf\{t\ge 0\colon X_t=0\}.\vspace{-0.2cm}
  \end{equation}
  Consider $(X,Z):=\big((X_t)_{0\le t\le T},(Z^i)_{0\le i\le J_T}\big)$, with $Z^i$ as the mark of the $i$-th jump of $X$. 
  Then the skewer process of $(X,Z)$ is an up-down oCRP$(\alpha,0)$ starting from $(n_0)\in\mathcal{C}$.
\end{thm} 

This extraction of a Markov process from the level sets of another process is in the spirit of the Ray--Knight theorems that identify the
local time process of suitably stopped Brownian motion as squared Bessel processes, see e.g.\ \cite{revuzyor99}.

We also generalise in three directions: first, to start from any $(n_1,\ldots,n_k)\in\mathcal{C}$, we concatenate $k$ independent copies of $(X,Z)$ replacing $n_0$ 
by $n_1,\ldots,n_k$, respectively. Second, to obtain an up-down oCRP$(\alpha,\theta)$ for $\theta>0$, we add a similar construction $\big((X_t)_{t<0},(Z^{-i})_{i\ge 1}\big)$ to provide 
left-most tables, as well as their ``children'' and further ``descendants''. See Figure \ref{Figure2}. Third, if we replace $Q_\alpha$ by other Q-matrices on $\mathbb{N}_0$ subject to conditions that
ensure appropriate absorption in 0, and if we appropriately relax the restriction that new tables always start from a single customer, we still obtain a Markovian skewer process, which
we refer to as a \em generalised up-down oCRP$(\alpha,\theta)$. \em See Section \ref{sec24} for details. We remark that laws of absorption times of continuous-time Markov chains
are called ``phase-type distributions.'' The associated JCCPs with phase-type jumps have been studied in other contexts \cite{AAP2004}. A feature of phase-type
distributions is that they are (weakly) dense in the space of all probability measures on $[0,\infty)$.

Returning to the setup of Theorem \ref{thm:skewer}, we establish distributional scaling limits for the table size evolution and the total number of customers in an up-down 
oCRP$(\alpha,\theta)$, as the initial number of customers tends to infinity, and for the L\'evy process $(X_t,t\ge 0)$. To formulate the limit theorems, let us discuss squared 
Bessel processes starting from $a\in[0,\infty)$ with dimension parameter $\delta\in\mathbb{R}$. We follow \cite{jaeschkeyor03,revuzyor99} and consider the unique strong solution 
of the stochastic differential equation (SDE)\vspace{-0.1cm}
\begin{equation}\label{SDE} dY(s)=\delta\, ds+2\sqrt{|Y(s)|}\,dW(s),\qquad Y(0)=a,\vspace{-0.1cm}
\end{equation}
driven by a standard Brownian motion $\big(W(s)\big)_{s\ge 0}$. For $\delta\ge 0$ this exists as a nonnegative stochastic process for all $s\ge 0$, never reaching 0 when $\delta\ge 2$, 
reflecting from 0 when $\delta\in(0,2)$ and absorbed in 0 when $\delta=0$. We denote this distribution by ${\tt BESQ}_a(\delta)$. For $\delta<0$, let 
$\zeta=\inf\{s\ge 0\colon Y(s)=0\}$ and denote by ${\tt BESQ}_a(\delta)$ the distribution of the process $\big(Y(s\wedge\zeta)\big)_{s\ge 0}$ that is absorbed at its first hitting time
of 0. 

\begin{thm}\label{thm:lim1} For all $0\le\alpha\le 1$, the table size evolution $Z_n$, i.e.\ a $Q_\alpha$-Markov chain, starting from $Z_n(0)=\lfloor nz\rfloor$, satisfies\vspace{-0.1cm}
  $$\left(\frac{Z_n(2ns)}{n}\right)_{s\ge 0}\stackrel{D}{\longrightarrow}\ {\tt BESQ}_z(-2\alpha)\qquad\mbox{under the Skorokhod topology.}\vspace{-0.1cm}$$
  Furthermore, the convergence holds jointly with the convergence of hitting times of 0.
\end{thm} 

We remark that the convergence of hitting times is not automatic since hitting times are not continuous functionals on Skorokhod space \cite{jacodshiryaev02}.

\begin{thm}\label{thm:lim2} For all $0\le\alpha\le 1$ and $\theta\ge 0$, the evolution $M_n$ of the total number of customers in an up-down oCRP$(\alpha,\theta)$, 
  i.e.\ a continuous-time Markov chain whose non-zero off-diagonal entries are $q_{m,m+1}=m+\theta$ and $q_{m,m-1}=m$, $m\ge 0$, starting from $M_n(0)=\lfloor na\rfloor$, 
  satisfies\vspace{-0.1cm}
  $$\left(\frac{M_n(2ns)}{n}\right)_{s\ge 0}\stackrel{D}{\longrightarrow}\ {\tt BESQ}_a(2\theta)\qquad\mbox{under the Skorokhod topology.}\vspace{-0.1cm}$$
\end{thm}

This appearance of squared Bessel processes further strengthens the connection of Theorem \ref{thm:skewer} to the Ray--Knight theorems.

\begin{thm}\label{thm:lim3} For all $0<\alpha<1$, the L\'evy process $(X_t)_{t\ge 0}$ of \eqref{eqn:Levyprocess} satisfies\vspace{-0.1cm}
  $$\left(\frac{X_{2n^{1+\alpha}t}}{2n}\right)_{t\ge 0}\stackrel{D}{\longrightarrow}\ {\tt Stable}(1+\alpha)\qquad\mbox{under the Skorokhod topology,}\vspace{-0.1cm}$$
  where ${\tt Stable}(1+\alpha)$ is the distribution of a spectrally positive stable L\'evy process with Laplace exponent 
  $\psi(\lambda)=\lambda^{1+\alpha}/2^\alpha\Gamma(1+\alpha)$.
\end{thm}

\begin{figure}\centering \includegraphics[trim = 0mm 0mm 0mm 0mm, clip, scale=0.6]{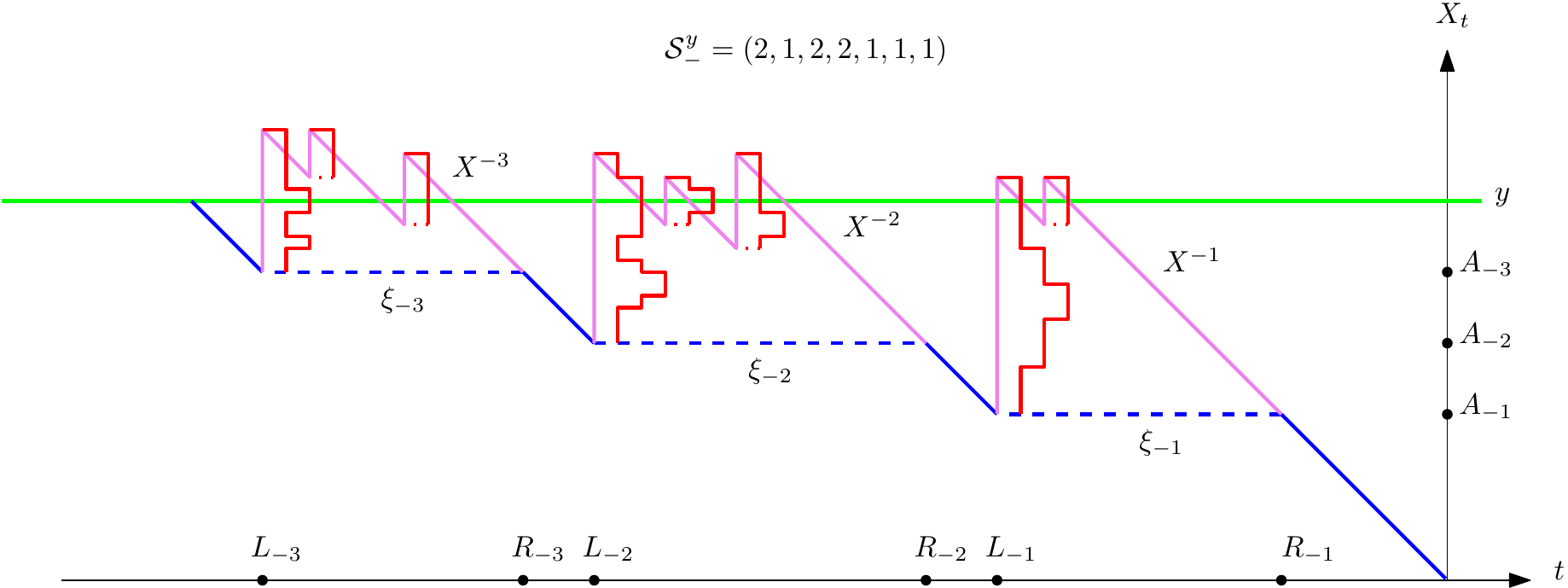}
\vspace{-0.2cm}
\caption[vt]{The skewer at level $y$ of the additional marked JCCP when $\theta>0$.\vspace{-0.2cm}}
\label{Figure2}
\end{figure}

This connects to work by Forman et al. \cite{fourauthor16,PartA,PartB}, where the starting point is a $\sigma$-finite ${\tt BESQ}(-2\alpha)$ excursion measure $\Theta$ due to
Pitman and Yor \cite{pitmanyor82}. Specifically, a 
${\tt Stable}(1+\alpha)$ L\'evy process $(\overline{X}_t)_{t\ge 0}$ is constructed from $\overline{Z}^0\sim{\tt BESQ}_z(-2\alpha)$ and a Poisson Random Measure 
$\overline{N}:=\sum_{t>0\colon \overline{Z}^t\not\equiv 0}\delta_{t,\overline{Z}^t}$ with intensity measure ${\rm Leb}\times\Theta$ by using the excursion lengths $\overline{\zeta}_t:=\inf\{s\ge 0\colon \overline{Z}_s^t=0\}$ as jump 
heights in a compensating limit\vspace{-0.1cm}
$$\overline{X}_t=\lim_{\varepsilon\downarrow 0}\left(\sum_{0\le r\le t\colon\overline{\zeta}_r>\varepsilon}\overline{\zeta}_r
                                                     \ -\ t\frac{1+\alpha}{2^\alpha\Gamma(1-\alpha)\Gamma(1+\alpha)}\varepsilon^{-\alpha}\right).\vspace{-0.1cm}$$ 
Furthermore, an interval partition diffusion is extracted from $(\overline{X},\overline{N})$ using a continuous analogue of the skewer process described above. This interval
partition diffusion is called a type-1 evolution. Type-0 evolutions are obtained by specifying a  semi-group that has further intervals added on the left-hand side in a way that 
achieves some left-right symmetry. Further generalisations will be explored in \cite{fprsw19}. 
Further  \pagebreak 
to the scaling limits of the encoding L\'evy processes and integer-valued up-down chains in Theorems \ref{thm:lim1}--\ref{thm:lim3}, we believe that the discrete composition-valued and interval-partition-valued processes satisfy the following scaling limit convergence.

\begin{conj}\label{conj2}
  Let $(\mathcal{S}_n^y)_{y\ge 0}$ be a $\mathcal{C}$-valued continuous-time up-down oCRP$(\alpha,\theta)$, for each $n\ge 1$. If $n^{-1}\mathcal{S}^0_n$ converges, as $n\rightarrow\infty$, then $(n^{-1}\mathcal{S}_n^{[ny]})_{y\ge 0}$ has a 
  diffusive scaling limit, when suitably represented on a space of interval partitions. For $0<\theta=\alpha<1$ and $0=\theta<\alpha<1$, these are the type-0 and type-1 
  evolutions of \cite{PartB}.
\end{conj}

\subsection{Structure}

The structure of this paper is as follows. In Section \ref{sec:skewer} we make rigorous the definition of the skewer process and prove Theorem \ref{thm:skewer} and its
generalisations. In Section \ref{sec:limits} we establish the scaling limit results Theorems \ref{thm:lim1}, \ref{thm:lim2} and \ref{thm:lim3}.

\section{Proof and generalisations of Theorem \ref{thm:skewer}}\label{sec:skewer}

In Section \ref{sec21} we establish the distribution of the hitting time of 0 of $Q_\alpha$-Markov chains, before making rigorous the skewer construction of the introduction in Section \ref{subsec:skewer}. In Section \ref{sec23} we prove Theorem \ref{thm:skewer}, and Section
\ref{sec24} discusses generalisations of Theorem \ref{thm:skewer} with general initial conditions in Corollary \ref{skewerj}, to skewer constructions 
of up-down oCRP$(\alpha,\theta)$ with $\theta>0$ in Theorem \ref{skewertheta}, and to more general up-down oCRPs in Theorems \ref{MarkovJCCP} and \ref{gen2}.

\subsection{The distribution of the absorption time of a $Q_\alpha$-Markov chain}\label{sec21}

As a first step towards the proof of Theorem \ref{thm:skewer}, let us show that the process $(X_t)_{t\ge 0}$ of \eqref{eqn:Levyprocess} is well-defined with $T<\infty$ almost surely. 
Since $(X_t)_{t\ge 0}$ has independent and identically distributed jumps at rate $\alpha$ and unit downward drift, this will follow if we show that the mean of the jump height 
distribution is $1/\alpha$. Jump heights are hitting times of $0$ by $Q_\alpha$-Markov chains starting from 1. Denote by $\mathbb{P}_r$ the distribution (on the Skorokhod space 
$\mathbb{D}([0,\infty),\mathbb{R})$ of c\`adl\`ag functions $g\colon[0,\infty)\rightarrow\mathbb{R}$) of a $Q_\alpha$-Markov chain starting from $r\in\mathbb{N}_0$. To study
scaling limits in Section \ref{sec:limits}, we actually need to identify the distribution of the hitting time $\zeta$ of 0 under $\mathbb{P}_1$.

We use a classical technique relating hitting times of birth-and-death processes and continued fractions. See e.g. Flajolet and Guillemin \cite{flajoletguillemin00} for related 
developments. \pagebreak

\begin{prop} \label{taudist}

Under $\PP_1$, the hitting time $\zeta$ of $0$ has Laplace transform given by\vspace{-0.2cm}
$$
\EE_{1}(e^{-\lambda \zeta}) = 1 - \frac{\lambda}{\alpha} + \frac{\lambda^{1+\alpha}}{\alpha e^\lambda \Gamma(1+\alpha,\lambda)}\quad\ \  \textrm{if } 0 < \alpha \leq 1,\vspace{-0.2cm}
$$
where the \emph{incomplete Gamma function} $\Gamma(a,z)$ is defined as $\Gamma(a,z):= \int_{z}^{\infty} e^{-t} t^{a-1} dt$, 
 for $z > 0$, $a \in \RR$. In particular, $\zeta$ is exponential with rate $1$ if $\alpha = 1$.  Moreover, $\EE_1(\zeta) = \alpha^{-1}$ if $0< \alpha \leq 1$.
\end{prop}


\begin{proof}

Define $\sigma_{r}:= \inf\{t \geq 0 \colon Z_t = r \}$ with the convention that $\inf \varnothing = \infty$. Suppose that $f_r$ is the density of $\sigma_r$ under $\PP_{r+1}$ and 
$g_{r}(t):= \PP_{r}\big(Z_t = r, Z_s \geq r \textrm{ for all } s\leq t\big)$.
By an application of the strong Markov property at the first jump, 
\begin{equation}\label{Markov}
g_{r}(t) =  e^{-(2r-\alpha)t}  + \int_{0}^{t} (r-\alpha) e^{-(2r-\alpha)u} \int_{0}^{t-u} f_{r}(s) g_{r}(t-s-u) ds du.
\end{equation}
For $\lambda > 0$, we define Laplace transforms
$$F_{r}(\lambda):= \int_{0}^{\infty} f_{r}(t)e^{-\lambda t} dt,\qquad G_{r}(\lambda):= \int_{0}^{\infty} g_{r}(t)e^{-\lambda t} dt,$$
and 
\begin{equation}\label{Hr} 
  H_{r}(\lambda):= \int_{0}^{\infty} e^{-(2r-\alpha)t} e^{-\lambda t} dt = (2r-\alpha + \lambda)^{-1}.
\end{equation} 
Now, by \eqref{Markov}, 
\begin{align*}
  (r\!-\!\alpha)F_{r}(\lambda)G_{r}(\lambda)H_{r}(\lambda) 
     &= (r\!-\!\alpha)\! \int_{0}^{\infty}\!\bigg{[}\int_{0}^{t} \int_{0}^{t-s} e^{-(2r-\alpha)t} f_{r}(s) g_{r}(t\!-\!s\!-\!u)  \bigg{]} du ds e^{-\lambda t} dt\\
     &= \int_{0}^{\infty} \big{(} g_{r}(t)  -  e^{-(2r-\alpha)t} \big{)} e^{-\lambda t} dt = G_{r}(\lambda) - H_{r}(\lambda).
\end{align*}
Rearranging this and evaluating $H_{r}(\lambda)$ as in \eqref{Hr}, it follows that 
$$G_{r}(\lambda) = \frac{1}{2r-\alpha+\lambda - (r-\alpha)F_{r}(\lambda)}.$$
Now for all $\varepsilon > 0$,
\begin{align*}
\int_{t}^{t+\varepsilon} f_{r}(s) ds &= \PP_{r+1}(t < \sigma_{r} \leq t+ \varepsilon)\\
&= \sum_{k=1}^{\infty} \PP_{r+1}\big(Z_t = r+k, Z_s \geq r+1 \textrm{ for all } 0 \leq s \leq t\big) \PP_{r+k}(\sigma_r \leq \varepsilon)\\
&= g_{r+1}(t) \PP_{r+1}(\sigma_r \leq \varepsilon) + o(\varepsilon^2) \quad\ \  \textrm{as } \varepsilon \downarrow 0.
\end{align*}
Upon dividing by $\varepsilon$ and letting $\varepsilon\downarrow 0$, we see that $f_{r}(t) = (r+1)g_{r+1}(t)$. Hence for all $r\geq 0$,
\begin{equation}\label{recursion}
F_{r}(\lambda) = (r+1)G_{r+1}(\lambda) = \frac{r+1}{2r+2-\alpha+\lambda - (r+1-\alpha)F_{r+1}(\lambda)}.
\end{equation}
Now let $a>1$, $z>0$, and  define $K_{-2}(\lambda):= e^{z}z^{-a}\Gamma(a,z)$ and $K_{-1}(\lambda)$ implicitly by 
\begin{align*}
K_{-2}(\lambda) = \frac{1}{1+z-a - (1-a)K_{-1}(\lambda)}.
\end{align*}
Then 
\begin{align*}
K_{-1}(\lambda) =\frac{(1+z-a)e^{z}z^{-a}\Gamma(a,z) - 1}{(1-a)e^{z}z^{-a}\Gamma(a,z)}= \frac{1+z-a}{1-a}- \frac{1}{(1-a)e^{z}z^{-a}\Gamma(a,z)}.
\end{align*}
However, it is known by \cite[p.278, 3.3.3]{fraction08} that for any $a>1$, $z >0$,
\begin{align*} 
K_{-2}(\lambda) = e^{z}z^{-a}\Gamma(a,z) = \cfrac{1}{1+z-a - \cfrac{1(1-a)}{3+z-a - \cfrac{2(2-a)}{5+z-a - \cfrac{3(3-a)}{7+z-a-\cdots}}}}\\
=  \cfrac{1}{1+z-a - (1-a)K_{-1}(\lambda)}\: .
\end{align*}
Upon using the recursion given for $F_{r}(\lambda)$ in \eqref{recursion}, we also see that
$$F_{0}(\lambda) =  \cfrac{1}{2+\lambda-\alpha - \cfrac{2(1-\alpha)}{4+\lambda-\alpha - \cfrac{3(2-\alpha)}{6+\lambda-\alpha - \cfrac{4(3-\alpha)}{8+\lambda-\alpha-\cdots}}}}\: .$$ 
Substituting the continued fraction expression for $K_{-2}(\lambda)$ into the expression for $K_{-1}(\lambda)$, it follows that 
$$K_{-1}(\lambda)= \cfrac{1}{3+z-a - \cfrac{2(2-a)}{5+z-a - \cfrac{3(3-a)}{7+z-a - \cfrac{4(4-a)}{9+z-a-\cdots}}}}\: .$$
As $\alpha \neq 0$, this can be identified with $F_{0}(\lambda)$ by setting $a = 1+\alpha, z = \lambda$, from which it follows that 
\begin{equation*}
F_{0}(\lambda)= K_{-1}(\lambda) = \frac{\lambda - \alpha}{-\alpha} - \frac{1}{-\alpha e^{\lambda} \lambda^{-(1+\alpha)}\Gamma(1+\alpha,\lambda)} = 1 - \frac{\lambda}{\alpha} + \frac{\lambda^{1+\alpha}}{\alpha e^\lambda \Gamma(1+\alpha,\lambda)}.
\end{equation*}
Since $F_{0}(\lambda)$ is the Laplace transform of $\zeta$ under $\PP_{1}$, by the definition of $f_{0}(\lambda)$, the proposition is proved upon observing that $\EE_{1}(\zeta) = -F'_{0}(0)$, and that for $\alpha=1$ the expression simplifies to the Laplace transform of the claimed exponential distribution.
\end{proof}

The fact that we obtain the exponential distribution when $\alpha=1$ is actually an elementary consequence of the observation that in this case $q_{1,0}=1$ and $q_{1,2}=0$.

\subsection{Construction of skewer processes in the setting of Theorem \ref{thm:skewer}}\label{subsec:skewer}

Consider the setting of Theorem \ref{thm:skewer}, i.e.\ for some fixed $0\le\alpha\le 1$, let $(J_t)_{t\ge 0}$ be a Poisson process of rate $\alpha$, and let $(Z^i,i\ge 0)$ be
an independent family of independent $Q_\alpha$-Markov chains starting from $Z^0(0)=n_0\ge 1$ for $i=0$ and from $Z^i(0)=1$ for $i\ge 1$, with absorption times 
$\zeta_i=\inf\{s\ge 0\colon Z^i(s)=0\}$, $i\ge 0$. By construction, the process $X_t:=-t+\sum_{0\le i\le J_t}\zeta_i$, $t\ge 0$, is a compound Poisson process with negative unit
drift, starting from $X_0=\zeta_0>0$. We use the convention $X_{0-}:=0$. While we will eventually stop $(X_t)_{t\ge 0}$ at $T:=\inf\{t\ge 0\colon X_t=0\}$, it will sometimes be useful to have access to the 
process with infinite time horizon. By Proposition \ref{taudist}, $(X_t)_{t\ge 0}$ is recurrent. As $(X_t)_{t\ge 0}$ has no negative jumps, $T<\infty$ almost surely, and so 

\begin{itemize}\item[(*)] $(X,Z):=\big((X_t)_{0\le t\le T},(Z^i)_{0\le i\le J_T}\big)$ is as follows. $X$ is c\`adl\`ag with finitely many jumps at times $U_i$, 
  $0\le i\le J_T$. For each jump, we have $\zeta_i:=X_{U_i}-X_{U_i-}>0$, and $Z^i=\big(Z^i(s)\big)_{0\le s<\zeta_i}$ is a positive integer-valued c\`adl\`ag path, $0\le i\le J_T$. 
\end{itemize}

We will define the skewer process of $(X,Z)$ as a function of $(X,Z)$ for any pair $(X,Z)$ that satisfies (*). See Figure \ref{Figure1} for an 
illustration. In this setting, first introduce notation for 
\begin{itemize}
  \item the pre- and post-jump heights $B_i:=X_{U_i-}$ and $D_i:=X_{U_i}$ of the $i$-th jump, which we will interpret as birth and death levels, $0\le i\le J_T$;
  \item the number $K^{y}:=\#\{0\!\le\! i\!\le\! J_T\colon B_i\!\leq\! y\!<\!D_i\}$ of jumps that cross level $y$, $y\!\ge\! 0$;
  \item for each level $y$ with $K^y\ge 1$, the indices $I_{1}^{y}:=\inf\{i\!\geq\! 0\colon B_{i}\!\leq\! y\!<\!D_{i}\}$ and 
    $I_{l}^{y}:=\inf\{i\!\geq\! I_{l-1}^{y}\!+\!1\colon B_{i}\!\leq\! y\!<\!D_{i}\}$, $2\le l\le K^y$, of jumps that cross level $y$;
  \item and the values $N^{y}_{l}:=Z^{I_{l}^{y}}(y-B_{I_l^{y}})$ of the integer-valued paths when crossing level $y$, i.e. evaluated $y-B_{I_l^y}$ beyond the birth level 
    $B_{I_l^y}$, for each $1\!\leq\! l\!\leq\! K^{y}$, $y\!\geq\! 0$.
\end{itemize}

\begin{defn}[Skewer process]
Using the notation above, define the \emph{skewer process} $(\Ss^{y})_{y\geq 0}$ of $(X,Z)$ by $\Ss^{y} := (N^{y}_1,N^y_2,\ldots,N^y_{K^y})$, $y\ge 0$.
\end{defn}
The skewer process introduced here is a discrete analogue of an interval-partition-valued process introduced in \cite{PartA} in a setting where $X$ may have a dense set 
of jump times, each marked by a non-negative real-valued path. Whether composition-valued or interval-partition-valued, the skewer process collects, in 
left to right order, the values of the paths at level $y$. This terminology stems from our visualisation in Figure \ref{Figure1}, where we imagine a skewer piercing the graph of 
the marked $X$-process at level $y$. The values encountered at level $y$ are pushed together without leaving gaps, as if on a skewer.

\subsection{Proof of Theorem \ref{thm:skewer}}\label{sec23}

\begin{figure}[t] \centering \includegraphics[trim = 0mm 0mm 0mm 0mm, clip, scale=0.6]{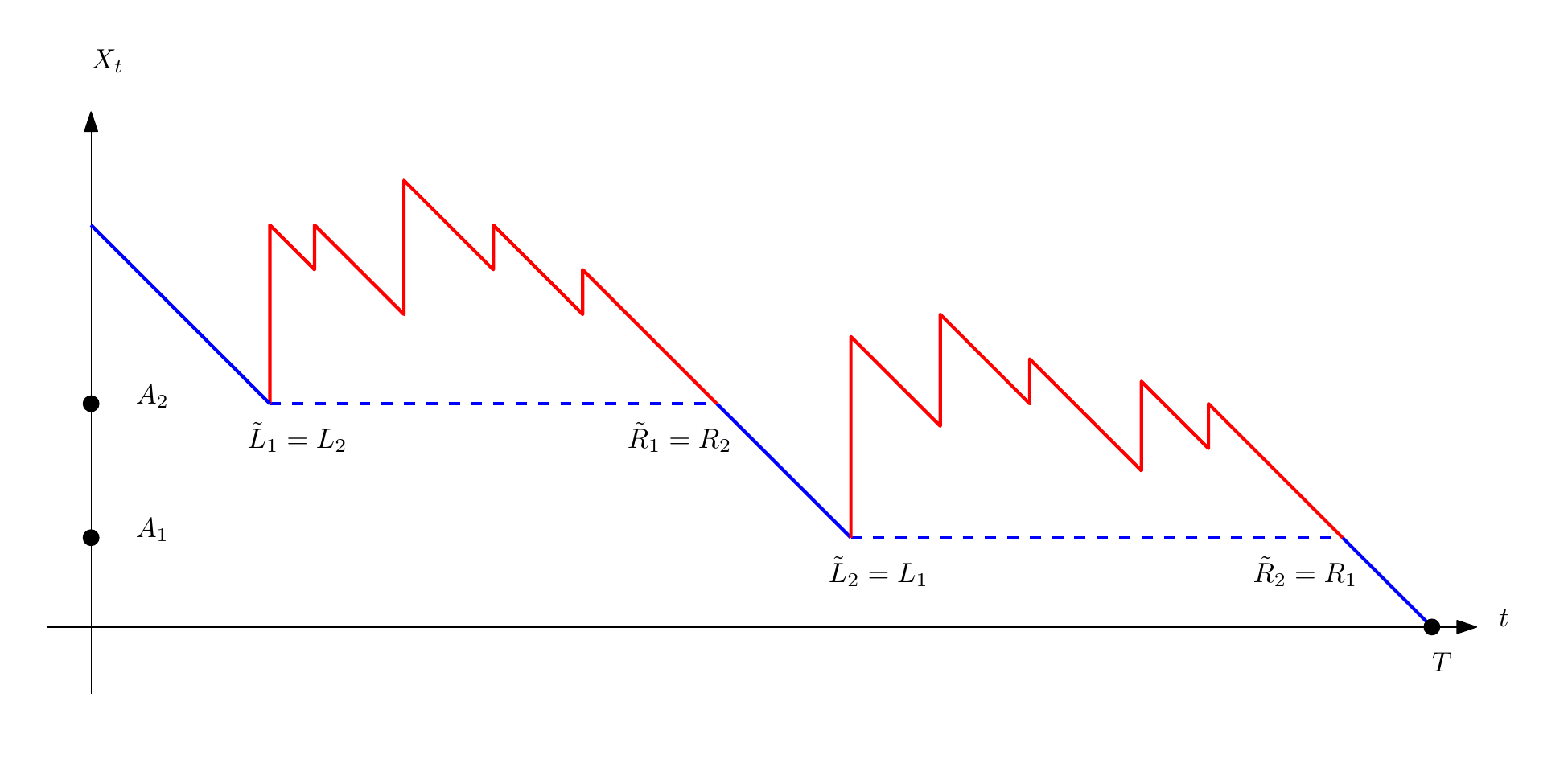}\vspace{-0.3cm}
\caption[yt]{An excursion of the JCCP with notation used in the proof of Theorem \ref{thm:skewer}.}
\label{Figure!!!}
\end{figure}

\begin{proof}[Proof of Theorem \ref{thm:skewer}]

Let $(X_t)_{t\geq 0}$ be the (unstopped) L\'{e}vy process constructed from $(Z^{i})_{i\geq 0}$ and $(J_{t})_{t\geq 0}$ as before via $X_{t}:=-t+\sum_{i=0}^{J_{t}}\zeta_{i}$, where 
$\zeta_i=\inf\{s\ge 0\colon Z^i(s)=0\}$, noting that $Z^0 \sim\PP_{n_0}$ and $Z^i \sim\PP_1$ for $i \geq 1$, independently, and that $(J_{t})_{t\geq 0}$ is an independent 
Poisson process of rate $\alpha$, a PP$(\alpha)$, for short. Denote the distribution of the marked process $(X,Z):=\big((X_t)_{0\le t\le T},(Z^i)_{0\le i\le J_T}\big)$, stopped at $T=\inf\{t\ge 0\colon X_t=0\}$, by $P_{n_0}$.

It may be useful to view $\zeta_\varnothing:=X_0$ as the ``lifetime'' of a table of initial size $n_0$. In a genealogical interpretation, where the individuals in the genealogy are tables and the children of each table are those inserted directly to the right at rate $\alpha$, this initial table would be the ancestor, hence notation $\varnothing$. We will not require notation for a full genealogical representation of the marked L\'evy process.
Instead, we proceed to decompose $(X,Z)$ under $P_{n_0}$ into the excursions above the minimum
that can be viewed as recursively capturing the genealogies of the children of the ancestor.

The skewer process makes the size of this initial table, varying according to $Z^0$, stay the left-most entry at all levels $0\le y<X_0$. Let us 
show that there is a PP$(\alpha)$ corresponding to the insertion of tables directly to the right of the initial table. By the definition of the skewer process, the $i$-th jump of 
$X$ corresponds to the second entry in the evolving composition (a table adjacent to the initial table) if its pre-jump level is a running minimum 
$X_{U_i-}=\inf_{0\le t\le U_i}X_t\ge 0$. These pre-jump levels are represented in Figure \ref{Figure!!!} as a collection of points on the vertical axis. 
Let $\tilde{R}_0:=0$ and define, also beyond $T$, the left and right endpoints of the excursions of $(X_t)_{t\ge 0}$ above the minimum\vspace{-0.1cm}
$$
\tilde{L}_{j}:= \inf\{t \geq \tilde{R}_{j-1}\colon X_t \neq X_{t-} \},\quad\tilde{R}_{j}:= \inf\{t \geq \tilde{L}_{j}\colon X_t = X_{\tilde{L}_{j}-} \},\quad \ \ \textrm{for } j \geq 1.\vspace{-0.2cm}
$$
Write $K_\varnothing:=\sup\{j \!\geq\! 0\colon \tilde{R}_{j} \!<\! T\}$ for the number of such excursions before $T$ and reverse order so that, for $1\!\le\! j\!\le\! K_\varnothing$, the interval
$(L_j,R_j):= (\tilde{L}_{K_{\varnothing}-j+1},\tilde{R}_{K_{\varnothing}-j+1})$ captures the $j$-th excursion in the order encountered by the skewer process as a process indexed by level 
$y\!\ge\! 0$. Write \vspace{-0.1cm} $\tilde{A}_{j} := X_0 \!-\! X_{\tilde{R}_j}$, for $j\! \geq\! 1$ and $A_j := X_{R_j}$ for any $1\!\leq\! j\! \leq\! K_{\varnothing}$, where $\tilde{A}_0:= 0$, for the starting 
levels of these excursions, respectively below $X_0$ and above 0.

We claim that $(\tilde{A}_j)_{j \geq 1}$ are the points of an unstopped PP$(\alpha)$ independent of the IID sequence of \em marked excursions \vspace{-0.1cm} \em  
$\big((X_{t+\tilde{L}_j}\!-\!X_{\tilde{R}_j})_{0\le t\le\tilde{R}_j-\tilde{L}_j},(Z^{\widetilde{I}_j+i})_{0\le i\le\widetilde{I}_{j+1}-\widetilde{I}_j-1}\big)$, where 
$\widetilde{I}_j\!:=\!J_{\widetilde{L}_j}$ identifies the first jump in the $j$-th excursion, $j\!\ge\! 1$. To show this, \vspace{-0.1cm} observe that $\tilde{A}_1\!=\! X_0 \!-\! X_{\tilde{R}_1} \!=\! X_0 \!-\! X_{\tilde{L}_{1}-}=\tilde{L}_{1}=U_1$ which is exponentially distributed with rate $\alpha$. For $j \geq 1$, apply the strong Markov property of the marked L\'evy process at times $\tilde{L}_j$ and 
$\tilde{R}_{j}$. Then the inter-arrival times $(\tilde{A}_j - \tilde{A}_{j-1})_{j \geq 1}$ are all IID ${\rm Exp}(\alpha)$ random variables independent of the IID sequence of 
marked excursions, as claimed. In particular, the $(\tilde{A}_j)_{j\geq 1}$ are the points of a PP$(\alpha)$ as claimed. 

Note that since $(X_0,Z^0)$ and $\big((X_t - X_0)_{t\geq 0},(Z^i)_{i\ge 1}\big)$ are independent, it follows that $(X_0,Z^0)$ is independent of the $(\tilde{A}_j)_{j \geq 1}$ and of the marked excursions as these can be defined solely in terms 
of $\big((X_t - X_0)_{t\geq 0},(Z^i)_{i\ge 1}\big)$. 
By time-reversal of the PP$(\alpha)$ at the independent time 
$\zeta_\nl=X_0$, we obtain the following. \vspace{-0.1cm}
\begin{enumerate}\item[(S)] Under $P_{n_0}$, we have $Z^0\sim\mathbb{P}_{n_0}$ with absorption time $\zeta_\varnothing$. Given $Z^0$, the $(A_j)_{1 \leq j \leq K_{\varnothing}}$ 
  are the jump times of a PP$(\alpha)$ restricted to $[0,\zeta_\varnothing]$, and given also $(A_j)_{1\le j\le K_\varnothing}$, the $K_\varnothing$ marked excursions 
  are conditionally IID, each with distribution $P_1$.
\end{enumerate}

We will now verify the jump-chain/holding-time structure of the skewer process by an induction on the steps of the jump chain. Note that the total number of steps $M$ of the 
jump chain is positive and finite almost surely. We denote by $Y_m$, $1\le m\le M$, the levels of those $M$ steps of the skewer process. We also set $Y_m=\infty$ for $m\ge M+1$.

Consider the inductive hypothesis that the first $m\wedge M$ steps satisfy the jump-chain/holding-time structure of an up-down oCRP$(\alpha,0)$, and conditionally given those 
states of the jump chain and holding time variables, and given that $\mathcal{S}^{Y_m}\!=\!(n_1,\ldots,n_k)$, the marked excursions of $(X,Z)$ above $Y_m$ are independent  
with distributions $P_{n_i}$, $1\!\le\! i\!\le\! k$. 

Assuming the inductive hypothesis for some $m\ge 0$ and $(n_1,\ldots,n_k)\in\mathcal{C}$, we note that if $Y_m=\infty$ or if $Y_m<\infty$ with $(n_1,\ldots,n_k)=\varnothing$ the induction proceeds trivially, while if $Y_m<\infty$ with $(n_1,\ldots,n_k)\neq\varnothing$, we may assume that there are independent exponential clocks
\begin{enumerate}\item[(i)] of rate $n_i-\alpha$, from the $Q_\alpha$-Markov chain $Z^0$ under $P_{n_i}$ as in (S), triggering a transition into state $(n_1,\ldots,n_{i-1},n_i+1,n_{i+1},\ldots,n_k)$, for each $1\le i\le k$,
  \item[(ii)] of rate $n_i$, from the $Q_\alpha$-Markov chain $Z^0$ under $P_{n_i}$ as in (S), triggering a transition into state $(n_1,\ldots,n_{i-1},n_i-1,n_{i+1},\ldots,n_k)$,\,or to $(n_1,\ldots,n_{i-1},n_{i+1},\ldots,n_k)$ if $n_i=1$,
    for each $1\le i\le k$,
  \item[(iii)] of rate $\alpha$, from the PP$(\alpha)$ under $P_{n_i}$ as in (S), triggering a transition into state $(n_1,\ldots,n_{i},1,n_{i+1},\ldots,n_k)$, for each $1\le i\le k$. 
\end{enumerate}
Denoting by $E_{m+1}$ the minimum of these $3k$ exponential clocks, the lack of memory property of the other exponential clocks makes the system start afresh, with either 
$\mathcal{S}^{Y_{m+1}}=\varnothing$ or with independent marked excursions above $Y_{m+1}=Y_m+E_{m+1}$. Specifically, in (iii), we note that the post-$E_{m+1}$ PP$(\alpha)$ is 
still PP$(\alpha)$, and is independent of the first, $P_1$-distributed excursion, as required for the induction to proceed.  
\end{proof}

\subsection{Generalisations of Theorem \ref{thm:skewer}}\label{sec24}

Ultimately, we will describe how to extend Theorem \ref{thm:skewer} to the $(\alpha,\theta)$-case with more general integer-valued Markov chains and started from a general state 
$(n_1,\ldots,n_k)\in\Cc$. We start by describing how to start from the general state $(n_1,\ldots,n_k)\in\Cc$ in the basic $(\alpha,0)$-case. To this end, define the concatenation $uv\in\Cc$ of 
$u,v \in \mathcal{C}$ with $u=(u_1,\ldots,u_n)$ and $v=(v_1,\ldots,v_m)$ by $uv:= (u_1,\ldots, u_n,v_1,\ldots, v_m)$. 

\begin{cor} \label{skewerj}
For $(n_1,\ldots,n_k)\in\mathcal{C}$, consider skewer processes $(\Ss^{j,y})_{y\geq 0}$ arising from independent copies of the $(X,Z)$ of 
Theorem \ref{thm:skewer}, with  $n_0$ replaced by $n_j$, so that $\Ss^{j,0}\!=\!(n_j)\!\in\!\Cc$, $1\!\le\! j\le\! k$. Then the process $(\Ss^{y})_{y\geq 0}$ obtained by concatenation $\Ss^y=\Ss^{1,y}\Ss^{2,y}\cdots\Ss^{k,y}$ evolves as an up-down 
oCRP$(\alpha,0)$ starting from $(n_1,\ldots,n_k)\in\Cc$.
\end{cor} 

\begin{rem}
\rm
It is also possible to derive the up-down oCRP$(\alpha,0)$ from a single path of the L\'{e}vy process $(X_t)_{t\ge 0}$. Specifically, consider the L\'{e}vy process up to time 
$T_k$, defined to be the time of the $k$-th down-crossing over $0$. Using the same notation as previously, but with $T$ replaced by $T_{k}$ in the definition of $K^{y}$, then  
conditional on the event that $N_{1}^{0} = n_1, N_{2}^{0} = n_{2}, \ldots, N_{k}^{0} = n_k$, the skewer $\Ss^{y}$ of $\big((X_t)_{0\le t\le T_k},(Z^i)_{0\le i\le J_{T_k}}\big)$ evolves as an 
up-down oCRP$(\alpha,0)$ starting from state $(n_1,n_{2},\ldots,n_k)\in\Cc$. A proof involves noting that the excursions above level $0$ are independent and it is then a 
consequence of Theorem \ref{thm:skewer}. We leave the details of this proof to the reader.
\end{rem}

As was mentioned somewhat briefly in the introduction, it is possible to derive the up-down oCRP$(\alpha,\theta)$ with $\theta > 0$ via a similar 
construction to the one given above. Our construction deals with the table insertions at rate $\theta$ by adding in excursions of $\big((X_t)_{t\geq 0},(Z^i)_{i\ge 0}\big)$ 
corresponding to the arrivals at rate $\theta$ of the new customers opening a new left-most table; this is visualised in Figure \ref{Figure2}.

More formally, suppose that we have the following objects, independent of each other and of the objects in the previous construction:
\begin{itemize}\item the times $(A_{-j})_{j\ge 1}$ of a Poisson process with intensity $\theta > 0$,
\item an IID sequence of $P_1$-distributed marked excursions. 
\end{itemize}
Denoting the excursions by $(X_r^{-j})_{0\le r\le\xi_{-j}}$, $j\ge 1$, we now define a process $(X_t)_{t<0}$ by inserting into a process of unit negative drift the $j$-th excursion
at level $A_{-j}$, $j\ge 1$. To this end, we let $L_0:=0$ and denote by $L_{-j}:=-A_{-j}-\sum_{1\le i\le j}\xi_{-i}$ and $R_{-j}:=L_{-j}+\xi_{-j}$ the left and right 
endpoints of the $j$-th of these excursions, $j\ge 1$, and then set for $t<0$,
$$
X_{t}:=\left\{
\begin{array}{ll}
A_{-j} + X^{-j}_{t-L^{-j}}&\textrm{ if }t \in [L_{-j},R_{-j}]\textrm{ for some }j\ge 1,\\[0.2cm]
|t|-\sum_{1\le i\le j-1}\xi_{-i}&\textrm{ if }t \in (R_{-j},L_{-j+1})\textrm{ for some }j\ge 1.\\
\end{array}
\right.
$$
Then the marks of the marked excursions naturally give rise to an infinite sequence $(Z^{-i})_{i\ge 1}$ marking the jumps of $(X_t)_{t<0}$. Note that we can recover the 
decomposition into marked excursions from $\big((X_t)_{t<0},(Z^{-i})_{i\ge 1}\big)$ by decomposing along the running minimum process $(\inf_{r\le t}X_r)_{t<0}$. Indeed, observe that this 
running minimum process has intervals $[L_{-j},R_{-j}]$ of lengths $\xi_{-j}$ along which it is constant at heights $A_{-j}$, with $L_{-j} < R_{-j}<L_{-j+1}$, $j\ge 1$. We 
denote the distribution of $\big((X_t)_{t<0},(Z^{-i})_{i\ge 1}\big)$ by $P^-$ and note that, by construction, we have a decomposition similar to (S):
\begin{itemize}\item[(S$^-$)] Under $P^-$, the $(A_{-j})_{j\ge 1}$ 
  are the jump times of a PP$(\theta)$, and the marked excursions 
  are independent of $(A_{-j})_{j\ge 1}$ and IID, each with distribution $P_1$.
\end{itemize}
Extending the notation of the bullet points in Section \ref{subsec:skewer}, we denote by $(U_{-i})_{i\ge 1}$ the sequence of jump times of $(X_t)_{t<0}$, by $B_{-i}:=X_{U_{-i}-}$ 
and $D_{-i}:=X_{U_{-i}}$ their pre- and post-jump heights, by $K_-^y:=\#\{i\ge 1\colon B_{-i}\le y<D_{-i}\}$ the number of jumps across level $y$ and for each level with 
$K_-^y\ge 1$, we let $I_{-1}^y:=\inf\{i\ge 1\colon B_{-i}\le y<D_{-i}\}$ and $I_{-l}^y:=\inf\{i\ge I_{-(l-1)}^y+1\colon B_{-i}\le y<D_{-i}\}$, $2\le l\le K_{-}^y$. Then we define
the skewer process of $\big((X_t)_{t<0},(Z^{-i})_{i\ge 1}\big)$ at level $y$ to be\vspace{-0.1cm}
\begin{equation}\label{eqn:negskewer}
\Ss_-^y:=(N^y_{-K_-^y},\ldots,N^y_{-1}),\qquad\mbox{where }N^y_{-l}:=Z^{-I_{-l}^y}(y-B_{-I_{-l}^y}),\ 1\le l\le K_-^y.\vspace{-0.1cm}
\end{equation}
This construction was illustrated in Figure \ref{Figure2}.

\begin{thm} \label{skewertheta} For any $0\le\alpha\le 1$, $\theta>0$ and $(n_1,\ldots,n_k)\in\Cc$, let $\big((X_t)_{t\ge 0},(Z^i)_{i\ge 0}\big)$ be as in Theorem \ref{thm:skewer} and, 
  independently, let $\big((X_t)_{t<0},(Z^{-i})_{i\ge 1}\big)$ be as above. Then, with $(\Ss^y)_{y\ge 0}$ defined as in Corollary \ref{skewerj} and $(\Ss^y_-)_{y\ge 0}$ defined as in 
  \eqref{eqn:negskewer}, the process $(\Ss_-^y\Ss^{y})_{y\geq 0}$ evolves as an up-down oCRP$(\alpha,\theta)$ starting from $(n_1,\ldots,n_k)\in\Cc$.
\end{thm}

%
%

\begin{proof} We extend the induction on the number of steps in the jump-chain/holding-time description of the skewer process in the proof of Theorem 
\ref{thm:skewer}. Specifically, the induction hypothesis provides the $P^-$-distributed marked process on the negative $t$-axis, and we encounter one more 
independent exponential clock in the induction step,
\begin{enumerate}\item[(iv)] of rate $\theta$, from the PP$(\theta)$ under $P^-$ as in (S$^-$), triggering a transition into state $(1,n_1,\ldots,n_k)$.
\end{enumerate}
As in case (iii), we note if case (iv) triggers the transition, the post-$E_{m+1}$ PP$(\theta)$ is still PP$(\theta)$, and independent of the first, $P_1$-distributed marked excursion as required.
\end{proof}

Finally, we generalise the $Q_\alpha$-Markov chains and their initial distributions. This is based on the observation that much of the proof of Theorems \ref{thm:skewer} and \ref{skewertheta} above 
does not rely on the specific transition rates of $Q_\alpha$-Markov chains nor on letting excursions start from $1$. This establishes the following result.  

\begin{thm}\label{MarkovJCCP} Let $0\le\alpha\le 1$ and $\theta\ge 0$. Suppose that $Q$ is a Q-matrix on $\NN_0$ and $p^-$ and $p$ are distributions on $\NN$, with the 
  following properties:
  \begin{itemize}\item $Q$ is such that $q_{0,0}=0$, i.e.\ $0$ is an absorbing state.
    \item $Q$ is such that $\PP_n(\zeta<\infty)=1$ for all $n\ge 1$; 
      i.e. absorption (at the hitting time $\zeta$ of 0) happens with probability 1 under any $\PP_n$ (the law of a  
      $Q$-Markov chain starting from $n\ge 1$).
      In particular, $(p^-,Q)$ is such that $\sum_{n\ge 1}p_n^-\PP_n(\zeta<\infty)=1$.
    \item $(p,Q)$ is such that $\sum_{n\ge 1}p_n\EE_n(\zeta)\le 1/\alpha$.
  \end{itemize}
  Then replacing $Q_\alpha$ by $Q$ and taking $Z^i(0)\sim p$, $i\ge 1$, in the setting of Theorem \ref{thm:skewer} and in the definition of $P_n$, $n\ge 1$, and replacing $P_1$ by 
  $P_{p^-}:=\sum_{n\ge 1}p_n^-P_n$ in $({\rm S}^{-})$, the concatenated skewer process $(\Ss_-^y\Ss^{y})_{y\geq 0}$, constructed as in Theorem \ref{skewertheta}, is Markovian.
\end{thm}

Note that for $\mu := \sum_{n= 1}^{\infty}p_n \EE_n(\zeta)$, the assumption that $\mu\leq 1/\alpha$ is a condition of (sub)-criticality and corresponds to the condition 
required for the excursions of the L\'{e}vy process to be almost surely finite, which in particular is necessary to avoid $(X_t)_{t\ge 0}$ drifting to $\infty$; we previously 
had $\mu=1/\alpha$ by Proposition \ref{taudist}. Specifically, we note that new tables inserted on the right arrive at rate $\alpha$ and have an average lifespan of $\mu$
independently of arrival times; we need $\mu\alpha$ to be less than or equal to the rate $1$ of downward drift. If this does not hold, then it is not necessarily possible to 
apply the Strong Markov property at time $\tilde{R}_1$ and the proof of Theorem \ref{skewertheta} breaks down. Indeed, informally, when $\tilde{R}_1=\infty$, the rate-$\alpha$
insertion of a table to the right is absent in the skewer process, but whether $\tilde{R}_1=\infty$ or not depends on the future of the skewer
process.

Note that, on the other hand, no such restriction is needed for $p^-$ since $p^-$ only affects the first jump in an excursion of the L\'evy process, while the remainder of the
excursion has jumps according to the absorption time under $\sum_{n\ge 1}p_n\PP_n$. 

We can describe precisely the class of skewer processes in such constructions, of which the up-down oCRP$(\alpha,\theta)$ is a special case. 
We continue to interpret a composition $(n_1,\ldots,n_k)\in\Cc$ as $k$ tables in a row, with $n_1,\ldots,n_k$ customers, respectively. 

\begin{defn}[Generalised up-down oCRP$(\alpha,\theta)$]\label{GCRP} Let $0\le\alpha\le 1$ and $\theta \geq 0$, and suppose that $Q$ is a Q-matrix on 
  $\NN_0$ with 0 as an absorbing state, and that $p^-$ and $p$ are distributions on $\NN$. Then the \emph{generalised up-down oCRP$(\alpha,\theta)$ with Q-matrix $Q$, table
  distribution $p$ and left-most table distribution $p^-$} is a continuous-time Markov process where from state $(n_1,\ldots, n_k)\in\Cc$ there are independent exponential clocks
\begin{itemize}
\item of rate $q_{n_i,l}$, for the arrival at table $i$ of $l-n_i$ new customers if $l>n_i$ or the departure from table $i$ of $n_i-l$ existing customers if $l<n_i$, and leading to a transition into state $(n_1,\ldots, n_{i-1},l, n_{i+1},\ldots, n_k)$ if $l\geq 1$, $l\neq n_i$, or into state $(n_1,\ldots, n_{i-1},n_{i+1},\ldots, n_k)$ if $l=0$, for each $l\neq n_i$, $1\leq i \leq k$,
\item of rate $\alpha p_l$ for the arrival of $l$ customers at a new table to the right of table $i$, and leading to state $(n_1,\ldots, n_{i},l, n_{i+1},\ldots, n_k)$, for each $l\ge 1$, $1\leq i \leq k$,
\item of rate $\theta p_l^-$ for the arrival of $l$ customers at a new table to the left of table 1, and leading to state $(l,n_1,\ldots,n_k)$, for each $l\ge 1$.
\end{itemize}

\end{defn}

The second and third bullet points can be rephrased as saying that at rate $\alpha$, respectively $\theta$, a group of random $p$-distributed size, respectively $p^-$-distributed size, arrives to open a new table to the right of each table $i$, $1\le i\le k$, respectively, to the left of table 1. 
The first bullet point allows to model the party size at each table by a pretty general $\mathbb{N}_0$-valued continuous-time Markov chain, with people of various group sizes coming and going until the party ends with the last people leaving the table.

As the reader may now expect, this process arises as the distribution of a skewer process in Theorem \ref{MarkovJCCP}: 

\begin{thm}\label{gen2}
  The skewer process in Theorem \ref{MarkovJCCP} is a generalised up-down oCRP$(\alpha,\theta)$ with Q-matrix $Q$, table distribution $p$ and left-most table distribution $p^-$.
\end{thm}

The proof follows by using essentially the same argument as for Theorems \ref{thm:skewer} and \ref{skewertheta}, and is therefore omitted.

\section{Scaling limits}\label{sec:limits}

In this section, we develop the scaling limit results stated as Theorems \ref{thm:lim1}, \ref{thm:lim2} and \ref{thm:lim3} in the introduction.  The first two results relate the chains governing table sizes and total numbers of customers in the up-down oCRP$(\alpha,\theta)$ to squared Bessel processes. The third result relates the L\'{e}vy process controlling the 
JCCP to a $(1+\alpha)$-stable L\'{e}vy process. Before proving these results in Section \ref{sec:proofs}, we cover preliminary material in Section \ref{sec:prel}. We discuss
some further consequences in Section \ref{sec:conseq}.

\subsection{Relevant material from the literature}\label{sec:prel}


\subsubsection{The Skorokhod topology and hitting times}\label{sec:Skor}
To discuss convergence of stochastic processes, we work with the Skorokhod topology on the space of c\`{a}dl\`{a}g paths. Specifically, 
let $(E,\rho)$ be a metric space and let $\DD([0,\infty),E)$ denote the space of c\`{a}dl\`{a}g functions from $[0,\infty)$ to $E$. Also denote by $\DD([0,t],E)$ the space of c\`adl\`ag functions from $[0,t]$ to $E$.

\begin{defn}[Skorokhod topology] The \emph{Skorokhod topology on $\DD([0,t],E)$} is the topology induced by the metric \vspace{-0.2cm}
$$d_{t}(f,g):= \inf_{\lambda \in \LL_t} \Big{(} \|\lambda\|^{\circ}_t \vee \|f - g \circ \lambda\|_t \Big{)},\vspace{-0.2cm} $$
where $\LL_t$ is the set of all strictly increasing bijections of $[0,t]$ into itself, $\|\cdot \|^{\circ}_t$ is the function on $\LL_t$ with
$\|\lambda\|_{t}^{\circ} := \sup_{r,s \colon 0\leq r <s \leq t} \big{|}\ln\big{(} \frac{\lambda(s)-\lambda(r)}{s-r}\big{)}\big{|}$, and $\|\cdot\|_{t}$ is the supremum norm on $\DD([0,t],E)$.

  The \emph{Skorokhod topology on $\DD(E,[0,\infty))$} is the topology induced by the metric \vspace{-0.2cm}
$$d(f,g):= \sum_{k=1} ^{\infty} \frac{1 \wedge d_{k}(f^{(k)},g^{(k)})}{2^k},\quad\textrm{where } f^{(k)}(t) := f(t) \mathbf{1}_{\{t < k-1\}} + (k-t) f(t) \mathbf{1}_{\{ k-1 \leq t \leq k\}}.
$$
\end{defn}
The following properties of the Skorokhod topology are used later.
\begin{lem} \label{Skoroprop}
The following results hold.
\begin{itemize} 
\item[(a)] \cite[Theorem 16.2]{billingsley99} $d(f_n,f) \rightarrow 0$ in $\DD([0,\infty),E)$ if and only if $d_{t}(f_n,f) \rightarrow 0$ for every continuity point $t$ of $f$. 
\item[(b)] \cite[Lemma 16.1, Theorem 16.2]{billingsley99} If $f_n \rightarrow f$ in $\DD([0,\infty),E)$ and $f$ is continuous, then $f_n \stackrel{\textrm{unif}}{\longrightarrow} f$ on $[0,t]$ for any $t>0$. 
\end{itemize}
\end{lem}

%

To show that convergence in distribution of hitting times holds, we establish the following lemma. 

\begin{lem} \label{taucty}
Consider the pair of functions $\tau,\tau^{<}\colon\DD([0,\infty),\RR)\rightarrow [0,\infty]$ given by 
$$\tau(f) := \inf\{t \geq 0 \colon f(t) \leq 0\}\qquad\mbox{and}\qquad \tau^{<}(f) := \inf\{t \geq 0 : f(t) < 0\},$$
and the function spaces $C([0,\infty),\RR) := \{f \in \DD([0,\infty),\RR)\colon f \textrm{ is continuous on }[0,\infty) \}$ and 
$\DD^{*}([0,\infty),\RR):= \big\{f \in \DD([0,\infty),\RR)\colon \forall t \geq \tau(f), f(t) \leq 0, \: \tau^{<}(f)= \tau(f)<\infty\big\}$.
If $f_n \rightarrow f$ in $\DD([0,\infty),\RR)$ and $f\in C([0,\infty),\RR)\cap \DD^{*}([0,\infty),\RR)$, then $\tau(f_n)\rightarrow \tau(f)$.

\end{lem}

\begin{proof} 

Suppose that $f_n \in \DD([0,\infty),\RR)$ converges under the Skorokhod topology to some function $f \in \DD^*([0,\infty),\RR) \cap C([0,\infty),\RR)$. For any $t \geq 0$, 
$d_{t}(f_n,f)\rightarrow 0$ by Lemma \ref{Skoroprop} (a), and as $f\in C([0,\infty),\RR)$, $f_n \stackrel{\textrm{unif}}{\longrightarrow} f$ on $[0,t]$ by 
Lemma \ref{Skoroprop} (b).

Let $\varepsilon > 0$. Then it follows from uniform convergence on $[0,\tau(f) - \varepsilon]$, the definition of $\tau$ and the fact that  
$\inf_{0 \leq t \leq \tau(f)-\varepsilon} f(t) >0$ that there exists $N_1 \in \NN$ such that for any $ n > N_1, f_n > 0 $ on $[0,\tau(f) - \varepsilon]$. This yields 
$\tau(f_n) > \tau(f) - \varepsilon$ for every $n > N_1$.

To show that $\tau(f_n) < \tau(f)+\varepsilon$ for sufficiently large $n$, recall that as $f\!\in\!\DD^*([0,\infty),\RR)$, $\tau(f)= \tau^{<}(f)$. By assumption, there is some 
$\delta \in (0,\varepsilon)$ with $f(\tau^{<}(f)+\delta) < 0$. Now  $f_n \stackrel{\textrm{unif}}{\longrightarrow} f$ on $[0,\tau^{<}(f)+\delta]$, and so there exists $N_2 \in\NN$
such that for any $n >N_2$, $f_{n}(\tau(f)+\delta)<0$. It follows that $\tau(f_n) \leq \tau(f)+\delta<\tau(f)+\varepsilon$.
Setting $N:= \max(N_1,N_2)$, we have shown that for all $n > N$, $|\tau(f)-\tau(f_n)| < \varepsilon$ as required.
\end{proof}

We will also use the following well-known convergence criteria for L\'evy processes. 

\begin{lem} e.g. \cite[Corollary VII.3.6]{jacodshiryaev02}\label{Jacodcadlagconvergence}
Suppose that $X^n$ and $\overline{X}$ are c\`{a}dl\`{a}g processes with stationary independent increments. Then $X^n \stackrel{D}{\rightarrow}\overline{X}$ on 
$\mathbb{D}([0,\infty),\RR)$ if and only if $X^{n}_{1}\stackrel{D}{\rightarrow}\overline{X}_{\!1} $. Moreover, whenever this holds and $F_{X^{n}}$ and $F_{\overline{X}}$ are 
the L\'{e}vy measures of $X^n$ and $\overline{X}$ respectively, then for every $g\in C_{2}([0,\infty),\RR)$, $F_{X^{n}}(g)\rightarrow F_{\overline{X}}(g)$, where $$C_{2}([0,\infty),\RR):=\big\{g\!\in\! C([0,\infty),\RR)\!\colon g \textrm{ is bounded and there is }\varepsilon\!>\!0 \textrm{ with } g\!=\!0 \textrm{ on } [0,\varepsilon)\big\}\!.$$
\end{lem}

\subsubsection{Squared Bessel processes}\label{sec:BESQ}

The scaling limit results of Theorems \ref{thm:lim1} and \ref{thm:lim2} involve $[0,\infty)$-valued squared Bessel processes of dimension $\delta\in\mathbb{R}$, which are absorbed at $0$ if $\delta\le 0$. As in the introduction, we denote these by ${\tt BESQ}_a(\delta)$. 
We will establish the convergence of absorption times by considering the 
natural extension of squared Bessel processes with negative dimension to the negative half-line.  

\begin{defn}[Extended squared Bessel process]\cite{jaeschkeyor03,revuzyor99} \label{defnBESQ} Let $\delta \in \RR$, $a\in\RR$, and suppose that $(W_s)_{s\geq 0}$ is a standard 
  Brownian motion in $\RR$. The unique strong solution of the SDE \eqref{SDE} is called an \emph{extended squared Bessel process with dimension $\delta$}. Its distribution on 
  $C([0,\infty),\RR)$ when starting from $a$ is denoted by ${\tt BESQ}_a^*(\delta)$.
\end{defn}


The drift parameter $\delta$ in \eqref{SDE} is referred to as the \em dimension \em since it can be shown for any integer $n \in\NN$, that if $(B_t)_{t\geq 0}$ is an $n$-dimensional Brownian motion, the evolution of the squared norm $\|B_t\|^2$, $t\ge 0$, is a ${\tt BESQ}_0(n)$-process.


We state some properties of squared Bessel processes needed later. 

\begin{prop}[Pathwise properties of squared Bessel processes] \label{pathBESQ}

Let $\delta \in \RR,a \geq 0$, and suppose that $Y^*$ is a ${\tt BESQ}_{a}^*(\delta)$ process.\vspace{-0.1cm}

\begin{itemize}
\item[(a)] \cite[p.330]{jaeschkeyor03} The process $-Y^*$ is a ${\tt BESQ}_{-a}^*(-\delta)$ process.\vspace{-0.1cm}
\item[(b)] \cite[Corollary 1]{jaeschkeyor03} If $\delta \geq 2$ and $a>0$, then $Y^*$ does not hit $0$ almost surely, and if $\delta < 2$ then $Y^*$ hits $0$ almost surely. \vspace{-0.1cm}
\item[(c)] \cite[Proposition XI.1.5]{revuzyor99} If $\delta \geq 0$ then $Y_t^* \geq 0$ for all $t\ge 0$, with $0$ an absorbing state if and only if $ \delta = 0$.
  If $\delta\in(0,2)$, then $0$ is instantaneously reflecting. 

\end{itemize}

\end{prop}

In particular, we have ${\tt BESQ}_a(\delta)={\tt BESQ}_a^*(\delta)$ for $\delta\ge 0$. For negative dimension $-\delta$, $\delta>0$, a ${\tt BESQ}_a(-\delta)$ process is by 
definition obtained as $Y=(Y^*_{s\wedge\zeta})_{s\ge 0}$ from a ${\tt BESQ}_a^*(-\delta)$ process $Y^*=(Y^*_s)_{s\ge 0}$, by stopping $Y^*$ at the first hitting time 
$\zeta:=\inf\{s\ge 0\colon Y^*_s=0\}$ of 0; by the Markov property of ${\tt BESQ}_a^*(-\delta)$ and by Proposition \ref{pathBESQ} (a), the process 
$\overline{Y}:=(-Y^*_{\zeta+s})_{s\ge 0}$ is then a ${\tt BESQ}_0(\delta)$ process independent of $Y$. Vice versa, we can replace the absorbed part of a ${\tt BESQ}_a(-\delta)$ 
process $Y$ by the negative of an independent ${\tt BESQ}_0(\delta)$ process $\overline{Y}$ as $Y_s^*=Y_s$, $s\le\zeta$, and $Y_{\zeta+s}^*=-\overline{Y}_s$, $s\ge 0$, to 
construct a ${\tt BESQ}^*_a(-\delta)$ process $Y^*$. A slight refinement of such arguments yields the following corollary, which uses notation from Lemma \ref{taucty}.

%

\begin{cor} \label{pathBESQ2}

When $\delta > 0$ and $a\geq 0$, almost all the paths of a ${\tt BESQ}_a^*(-\delta)$ process lie in $C([0,\infty),\RR)\cap \DD^{*}([0,\infty),\RR)$.

\end{cor}

\begin{proof}
Let $f$ be a path of ${\tt BESQ}_a^*(-\delta)$. By Proposition \ref{pathBESQ} (b), $f$ hits $0$ almost surely. By Proposition \ref{pathBESQ} (a), $-f$ is a 
${\tt BESQ}_{-a}(\delta)$ path. By Proposition \ref{pathBESQ} (c), applied to $-f$,\vspace{-0.1cm}
\begin{itemize}
\item $f \leq 0$ on $[\tau(f),\infty)$ a.s.\ if $-\delta \in (-2,0)$,\vspace{-0.2cm}
\item $f < 0$ on $(\tau(f),\infty)$ a.s.\ if $-\delta \leq -2$.\vspace{-0.1cm}
\end{itemize}

It is clear by a.s.\ continuity of $f$ that $f\big(\tau(f)\big) = 0$, since $f > 0$ on $[0,\tau(f))$. To show that $\tau^{<}(f) = \tau(f)$, suppose for a contradiction that this was not the case. Since $f \leq 0$ on $[\tau(f),\infty)$, $\tau^{<}(f) > \tau(f)$ would imply that $f =0$ on $[\tau(f), \tau^{<}(f))$. This contradicts part (c) of Proposition \ref{pathBESQ}, and so $\tau^{<}(f) = \tau(f)$. Hence $f \in \DD^{*}([0,\infty),\RR)$ a.s.
\end{proof}

By Proposition \ref{pathBESQ} (c), we have  $\tau^{<}(f)=\infty$ for paths $f$ of a ${\tt BESQ}_a^*(\delta)$, $\delta\ge 0$. In particular the conclusion of Corollary \ref{pathBESQ2} fails for such parameters.\pagebreak

\subsubsection{Martingale problems for diffusions}\label{sec:martprob}

To discuss some of the general theory needed for a proof later, we follow \cite{jacodshiryaev02}. In what follows, $\big(\Omega, (\Ff_t)_{t\geq 0},\Ff,\PP\big)$ is a filtered probability space and $\mathbf{F} = (\Ff_t)_{t\geq 0}$. 

\begin{defn}[Continuous semi-martingale] A \emph{continuous semi-martingale} $X=(X_t)_{t \geq 0}$ is a real-valued stochastic process on $(\Omega,\Ff,\PP)$ of the form $X_t= X_0 +M_t + B_t$, where $X_0$ is $\Ff_0$-measurable, $M$ a continuous local $\mathbf{F}$-martingale and $B$ a continuous $\mathbf{F}$-adapted process of finite variation, with $M_0\!=\!B_0\!=\!0$. The \emph{quadratic variation} $C = \langle M \rangle$ of $X$ is the unique continuous increasing $\mathbf{F}$-adapted process such that $M_{t}^{2} \!-\! \langle M \rangle_t$, $t\!\ge\! 0$, is a local martingale. We refer to $(B,C)=(B_t,C_t)_{t \geq 0}$ as the \emph{characteristics} of $X$.

If $\eta$ is any distribution on $\RR$, we say the \emph{martingale problem} $s\big(\sigma(X_0),X\,|\,\eta ; B,C\big)$ has a solution $P$ if $P$ is a probability measure on $(\Omega,\Ff)$ with $P(X_0 \in\,\cdot\,) = \eta$, and $X$ is a continuous semi-martingale on $(\Omega,\mathbf{F},\Ff,P)$ with characteristics $(B,C)=(B_t,C_t)_{t \geq 0}$.

\end{defn}

\begin{defn}[Homogeneous diffusion] A \emph{homogeneous diffusion} is a continuous semi-mart\-ingale $X$ on $(\Omega,\mathbf{F},\Ff,\PP)$ for which there exist Borel functions 
  $b\colon\RR\rightarrow\RR$ and $c\colon\RR\rightarrow[0,\infty)$ called the \emph{drift coefficient} and \emph{diffusion coefficient} such that $B_{t} = \int_{0}^{t} b(X_s) ds$ 
  and $C_{t} = \int_{0}^{t} c(X_s) ds$, $t\ge 0$.
\end{defn}

\rm
Note in particular that a ${\tt BESQ}^*_a(\delta)$ process is a homogeneous diffusion with drift coefficient $b(x) = \delta$ and diffusion coefficient $c(x) =4|x|$, 
$x\in\mathbb{R}$. This can be seen from the SDE characterisation in Definition \ref{defnBESQ} and \eqref{SDE}.

We will need two results from \cite{jacodshiryaev02}.
\begin{lem} \cite[Theorem III.2.26]{jacodshiryaev02} \label{BESQmg} Consider a homogeneous diffusion $X$ and let our notation be as above. Then the martingale problem $s\big(\sigma(X_0),X\, |\,\eta ; B,C\big)$ has a unique solution if and only if the SDE $dY_t = b(Y_t) dt + \sqrt{c(Y_t)} dW_t$, $Y_0 = \xi$, has a unique weak solution, for $(W_t)_{t\geq 0}$ an $\mathbf{F}$-Brownian motion and $\xi$ an $\Ff_0$-measurable random variable with distribution $\eta$. 

\end{lem}




\begin{lem}\cite[Theorem IX.4.21]{jacodshiryaev02} \label{limitthm} Suppose that for any $a \in \RR$, the martingale problem $s\big(\sigma(X_0), X\,|\,\delta_a;B,C\big)$ has a unique 
  solution $P_a$, and that $a \mapsto P_a(A)$ is Borel measurable for all $A \in \Ff$. For each $n\ge 1$, let $X^{n}$ be a Markov process started from an initial distribution $\eta^n$, where 
  $X^{n}$ has generator $\Ll^{n}$ of the form $$\Ll^{n}f(x) = \int_{y\in\RR}\big{(}f(x+y) - f(x)\big{)}K^{n}(x,dy),$$ for a finite kernel  $K^n$ on $\RR$.
  Suppose further that the functions $b^{n}$ and $c^{n}$ given by 
\begin{center}
\vspace{-0.5cm}
\begin{equation*}
b^{n}(x):= \int_{y\in\RR}y K^{n}(x,dy)\quad \textrm{and}\quad c^{n}(x):= \int_{y\in\RR}y^2 K^{n}(x,dy),\quad x\in\RR,
\end{equation*}
\end{center}
are well-defined and finite, and that the drift and diffusion coefficients $b$ and $c$ and the initial distribution $\eta$ of $X$ are such that, as $n\rightarrow\infty$, 

\begin{itemize}

\item[(a)]$b^n \rightarrow b$ and $c^n \rightarrow c$ locally uniformly,
\item[(b)]$\sup_{\{x\colon |x| \leq r\}} \int_{y\in\RR} |y|^2 \mathbf{1}_{\{|y| > \varepsilon \}}K^{n}(x,dy)  \rightarrow 0$ for any $\varepsilon > 0$ and any $r > 0$,
\item[(c)]$\eta^{n}\rightarrow \eta$ weakly on $\RR$.
\end{itemize}
Then $X^{n}$ converges in distribution to $X$ as $n\rightarrow\infty$ under the Skorokhod topology on $\DD([0,\infty),\RR)$.

\end{lem}

\subsection{Proofs of Theorems \ref{thm:lim1}, \ref{thm:lim2} and \ref{thm:lim3}}\label{sec:proofs}

Recall that Theorem \ref{thm:lim1} claims that the distributional scaling limit of the table size evolution of an up-down oCRP$(\alpha,\theta)$ is ${\tt BESQ}(-2\alpha)$ and that this convergence holds jointly with
the convergence of absorption times in 0.

\begin{proof}[Proof of Theorem \ref{thm:lim1}]
Let $0\le\alpha\le 1$ and $a\in\RR$. Observe by Lemma \ref{BESQmg} that the martingale problem associated to the SDE characterising the ${\tt BESQ}_a(\delta)$ diffusion has a unique solution $P_a$. 
Showing that $a\mapsto P_a(A)$ is Borel measurable is fairly elementary, and so we leave this to the reader.
Consider an extension $Z_n^*$ of the table size Markov chain $Z_{n}$ to $\mathbb{Z}$ that further transitions from $0$ into $-1$ at rate $\alpha$, and from state $-m$, $m \geq 1$, to $-m-1$ at rate $m$ and into $-m+1$ at rate $m-\alpha$. This chain is an instance of a Markov chain
on $\RR$ with generator $\Ll$ given by
$$\Ll f(x) = \big{(}(|x|\vee\alpha) -\alpha\big{)} \big(f(x+1)-f(x)\big) + \big(|x|\vee \alpha\big{)} \big{(}f(x-1) - f(x)\big{)}, \ \ x \in\RR.$$

Let $X^n:=\big(n^{-1}Z^*_n(2ns),s\ge 0\big)$. Then $X^n$ has generator 
$$\Ll^n f(x) = \int_{y\in\RR} \big(f(x+y)-f(x)\big) K^{n}(x,dy),$$ 
where the increment kernel $K^n (x,dy)$ is given by
$$K^{n}(x,dy)=n\Big(\big(\left(2n|x|\vee2\alpha\right)-2\alpha\big) \: \delta_{1/n}(dy) +\left(2n|x|\vee2\alpha\right)\: \delta_{-1/n}(dy)\Big),\quad x \in \RR.$$
Now
\begin{align*}\quad\int_{y\in\RR} y K^{n}(x,dy) &= (2n|x|\vee2\alpha)-2\alpha -(2n|x|\vee2\alpha) = -2\alpha,\\
\textrm{and}\ \int_{y\in\RR} y^2 K^{n}(x,dy) &= \frac{2n|x|\vee2\alpha - 2\alpha}{n} + \frac{2n|x|\vee2\alpha}{n}= 4\left(|x|\vee\frac{\alpha}{n}\right) - \frac{2\alpha}{n}\rightarrow 4|x|.
\end{align*}
This convergence is uniform in $x\in\RR$ under the supremum norm $\|\cdot \|$ since 
$$\left\|\: 4\left(|x|\vee\frac{\alpha}{n}\right) - \frac{2\alpha}{n} - 4|x|\: \right\|
  = \left\|\: 4|x| + \frac{2\alpha}{n}-4\left(|x|\vee\frac{\alpha}{n}\right) \: \right\| = \frac{2\alpha}{n} \rightarrow 0.$$
Moreover, for any $r \geq 0$, $\varepsilon >0$ and $n$ sufficiently large, $$\sup_{\{x:|x| \leq r\}} \int_{y\in\RR} |y|^2 \mathbf{1}_{\{|y| > \varepsilon \}} K^{n}(x,dy) =0.$$ 
It is clear that as $n\rightarrow\infty$, $n^{-1}\lfloor{nz}\rfloor\rightarrow z$. This shows that all the assumptions of Lemma \ref{limitthm} are satisfied.
It follows that the law of $X^n$ converges weakly in the Skorokhod topology on $\DD([0,\infty),\RR)$ to that of the diffusion process with drift coefficient $-2\alpha$ and diffusion coefficient $4|x|$, namely a ${\tt BESQ}_z^*(-2\alpha)$ process. 

Let $0<\alpha\le 1$. Then the paths of a ${\tt BESQ}^*_z(-2\alpha)$ process lie in $\DD^*([0,\infty),\RR)\cap C([0,\infty),\RR)$ by Corollary \ref{pathBESQ2} since $-2\alpha<0$. 
By applying Skorokhod's representation theorem on the separable space $\DD([0,\infty),\RR)$, we may assume that the convergence $X^n\rightarrow X$ that we have just established 
holds almost surely. By Lemma \ref{taucty}, this entails the convergence $\tau(X^n) \stackrel{a.s}{\longrightarrow} \tau(X)$. Therefore 
$\big(X^n, \tau(X^n)\big) \stackrel{a.s}{\longrightarrow} \big(X, \tau(X)\big)$, and this entails the claimed joint distributional convergence since $\tau(X^n)$ and $\tau(X)$ are the first hitting
times of 0 for the up-down chains $X^n$ and for the continuous $X$, none of which can reach the negative half-line without first visiting 0.

If instead $\alpha = 0$, then we cannot apply Corollary \ref{pathBESQ2}. Observe that the Markov chain $Z^*_n$ is a simple birth-death process. By a standard reference such as 
\cite[Corollary 6.11.12]{grimmettstirzaker01}, the hitting time $\tau(Z_n^*)$ of $0$ by a simple birth-death process $Z_n^*$ started from height $m \in \NN_0$ satisfies 
$$\PP\big(\tau(Z_n^*)\leq t\big) = \left(\frac{t}{t+1}\right)^m\quad\Rightarrow\quad\PP\big(\tau(X^n) \leq t \big) = \left(\frac{2nt}{2nt+1}\right)^{\lfloor{nz}\rfloor}.$$ 
We conclude that $\PP\big(\tau(X^n)\leq t\big)\rightarrow e^{-z/2t}$, as $n\rightarrow\infty$.

By \cite[p.319]{jaeschkeyor03}, it is known that $\zeta\stackrel{D}{=} z/2G$ for the absorption time $\zeta$ of a ${\tt BESQ}_z(\delta)$ and $G\sim {\tt Gamma}\big((2-\delta)/2,1\big)$ when $\delta < 2$. Hence 
$\PP(\zeta \leq t) =\PP\big({\tt Exp}(1) \geq z/2t\big) = e^{-z/2t}$ in the case $\delta=0$. Thus we still have $\tau(X^n)\stackrel{D}{\rightarrow}\tau(X)$, in the case when $\alpha = 0$. 

We need to strengthen this to joint 
convergence in distribution: since the distributions of pairs $\big(X^n,\tau(X^n)\big)$, $n\ge 1$, are tight, it suffices to show that any subsequential
distributional limit for the pairs is such that the limiting time is the extinction time of the limiting process. Specifically, if $\big(X^{n_k},\tau(X^{n_k})\big)$ converges, we may assume convergence holds almost surely, by Skorokhod's representation theorem, and the limit
$(\overline{X},\overline{\tau})$ is such that $\overline{\tau}\stackrel{D}{=}\tau(\overline{X})$, as the marginal distributions converge. Also,
we can use deterministic arguments as in the proof of Lemma \ref{taucty} to show that $\overline{\tau}\ge\tau(\overline{X})$ almost surely,
and with $\overline{\tau}\stackrel{D}{=}\tau(\overline{X})$, this yields $\overline{\tau}=\tau(\overline{X})$ almost surely, as required. 
Specifically, let $\varepsilon>0$. Since $X^{n_k}\rightarrow\overline{X}$ and $\tau(\overline{X})<\infty$ almost surely, we have
$\inf_{0\le t\le\tau(\overline{X})-\varepsilon}X^{n_k}(t)>0$ for sufficiently large $k$, i.e.\ $\tau(X^{n_k})\ge\tau(\overline{X})-\varepsilon$.
Since $\tau(X^{n_k})\rightarrow\overline{\tau}$ almost surely, this yields $\overline{\tau}\ge\tau(\overline{X})-\varepsilon$, and so $\overline{\tau}\ge\tau(\overline{X})$ almost surely, as $\varepsilon>0$ was arbitrary. 

To prove the original claims from those involving the modified processes, simply observe that $Z_n(s) = Z_n^*\big(s\wedge\tau(Z_n^*)\big)$, $s\ge 0$.
\end{proof}

Recall that Theorem \ref{thm:lim2} claims that the scaling limit of the evolution of the total number of customers in an up-down oCRP$(\alpha,\theta)$ is ${\tt BESQ}(2\theta)$.

\begin{proof}[Proof of Theorem \ref{thm:lim2}]
The convergence of the scaled evolution of the total number of customers to the ${\tt BESQ}(2\theta)$ diffusion is proved in the same way as Theorem \ref{thm:lim1}, where the extension of the integer-valued Markov chain $M_n$ to 
state space $\RR$ can be chosen to have generator 
$$\Ll f(x) := \big{(}|x|+\theta\big{)}\big{(}f(x+1)- f(x)\big{)} + |x|\big{(}f(x-1) - f(x)\big{)}$$ 
instead. Note that this Markov chain shares the relevant parts of the boundary behaviour of ${\tt BESQ}(2\theta)$ at 0 in that 0 is absorbing for $\theta=0$, while upward transitions from 0 are possible when $\theta>0$. The proof proceeds as above.  
\end{proof}


Let us turn to the proof of Theorem \ref{thm:lim3}. Specifically, we show that the L\'evy process $(X_t)_{t\ge 0}$ underlying the JCCP of the genealogy of the up-down 
oCRP$(\alpha,0)$ as in \eqref{eqn:Levyprocess} has a $(1+\alpha)$-stable L\'{e}vy process as its scaling limit, if $0< \alpha < 1$. 

%
%

\begin{proof}[Proof of Theorem \ref{thm:lim3}]   
Recall \eqref{eqn:Levyprocess}. Since $X_0/2n=\zeta_0/2n\rightarrow 0$ almost surely, we may assume in the sequel that $\zeta_0=0$ so that $(X_t)_{t\ge 0}$ is a L\'evy process starting
from $X_0=0$. By conditioning on the number $J_t$ of jumps in the time interval $(0,t]$, applying the independence of the jump heights $\zeta_i$ and using 
Proposition \ref{taudist}, 
\begin{align*}
\EE\left(e^{-\lambda X_t}\right) &= e^{\lambda t}\EE\left(\exp\left(-\lambda\sum_{i=1}^{J_t} \zeta_i\right)\right) = e^{\lambda t}\sum_{k=0}^{\infty} \left(\EE_{1}\left(e^{-\lambda \zeta}\right)\right)^k \PP(J_t = k)\\
&=e^{\lambda t}\sum_{k=0}^{\infty} \left(1 - \frac{\lambda}{\alpha} + \frac{\lambda^{1+\alpha}}{\alpha e^\lambda \Gamma(1+\alpha,\lambda)}\right)^k \frac{e^{-\alpha t}(\alpha t)^k}{k!}\\
&= e^{(\lambda-\alpha) t}\sum_{k=0}^{\infty} \left(\alpha t - \lambda t + \frac{\lambda^{1+\alpha} t}{e^\lambda \Gamma(1+\alpha,\lambda)}\right)^k \frac{1}{k!}=\exp\left(\frac{\lambda^{1+\alpha} t}{e^\lambda \Gamma(1+\alpha,\lambda)}\right).
\end{align*}
Hence $(X_t)_{t\ge 0}$ has Laplace exponent given by 
\begin{align*}
\varphi(\lambda) &= \frac{\ln \EE(e^{-\lambda X_t}) }{t} = \frac{\lambda^{1+\alpha} e^{-\lambda}}{\Gamma(1+\alpha,\lambda)}.
\end{align*}
Now note that the Laplace exponent of $(2n)^{-1}X_{2n^{1+\alpha}t}$ satisfies
\begin{align*}
\frac{\ln \EE\left(\exp\left(-\lambda (2n)^{-1}X_{2n^{1+\alpha}t}\right)\right)}{t} &= \frac{(\lambda/2n)^{1+\alpha} e^{-\lambda/2n}}{\Gamma(1+\alpha,\lambda/2n)} 2n^{1+\alpha}\\
  & = \frac{\lambda^{1+\alpha} e^{-\lambda/2n}}{2^{\alpha}\Gamma(1+\alpha,\lambda/2n)} \rightarrow \frac{\lambda^{1+\alpha}}{2^{\alpha}\Gamma(1+\alpha)}, \qquad \textrm{as } n\rightarrow \infty.
\end{align*}
By the convergence theorem for Laplace transforms, we have $(2n)^{-1}X_{2n^{1+\alpha}} \stackrel{D}{\rightarrow} \overline{X}_1$ as $n\rightarrow\infty$. By independence and 
stationarity of increments, we conclude by Lemma \ref{Jacodcadlagconvergence}: 
$$\left(\frac{X_{2n^{1+\alpha}t}}{2n}\right)_{t\geq 0} \stackrel{D}{\longrightarrow}\left(\overline{X}_t\right)_{t\geq 0}\qquad\mbox{ under the Skorokhod topology}$$ 
on $\DD([0,\infty),\RR)$, as $n\rightarrow\infty$.
\end{proof}

\subsection{Some consequences of the main results, and further questions}\label{sec:conseq}

In this section we explore further the connections between ${\tt BESQ}(-2\alpha)$ and ${\tt Stable}(1+\alpha)$ processes. 

\begin{prop}\label{stablelaw}

Let $\alpha\in(0,1)$, and denote by $Q_{y}^{-2\alpha}$ the law of a ${\tt BESQ}_{y}(-2\alpha)$ process started from $y>0$. If 
$\zeta$ denotes the absorption time of this process under $Q_y^{-2\alpha}$, then as $y \downarrow 0$, we have that $y^{-(1+\alpha)}Q_{y}^{-2\alpha}(\zeta > s)$ converges 
to $\widetilde{\Pi}(s,\infty)$, where $$\widetilde{\Pi}(ds) = \frac{1}{2^{1+\alpha}\Gamma(1+\alpha)}s^{-(2+\alpha)}ds$$ is the L\'{e}vy measure of a spectrally positive ${\tt Stable}(1+\alpha)$  
L\'{e}vy process.

\end{prop}

\begin{proof}

Recall from \cite[p.319]{jaeschkeyor03} that under $Q_{y}^{-2\alpha}$, $\zeta\stackrel{D}{=} y/2G$, where $G\sim {\tt Gamma}(1+\alpha,1)$ has shape parameter $1+\alpha$ and rate parameter $1$. Observe that
\begin{equation*}
Q_{y}^{-2\alpha}(\zeta > s) = \PP\left(G < \frac{y}{2s}\right)= \int_{0}^{y/2s}\frac{1}{\Gamma(1+\alpha)}e^{-u} u^{\alpha} du.
\end{equation*}
Substituting $x = 2su/y$, we see that 
\begin{equation*}
y^{-(1+\alpha)}Q_{y}^{-2\alpha}(\zeta > s) = \int_{0}^{1} \frac{1}{\Gamma(1+\alpha) 2^{1+\alpha}} e^{-yx/2s} x^{\alpha} s^{-(1+\alpha)} dx.
\end{equation*}
By letting $y \downarrow 0$ and using monotone convergence, the right hand expression converges to
\begin{equation*}
\int_{0}^{1} \frac{1}{\Gamma(1+\alpha) 2^{1+\alpha}} x^{\alpha} s^{-(1+\alpha)} dx = \frac{1}{2^{1+\alpha}\Gamma(2+\alpha)}s^{-(1+\alpha)}.
\end{equation*}
Writing $\widetilde{\Pi}(s,\infty) = \int_{s}^{\infty}\widetilde{\Pi}(du)$, we observe that indeed $y^{-(1+\alpha)}Q_{y}^{-2\alpha}(\zeta > s)\rightarrow\widetilde{\Pi}(s,\infty)$ as $y \downarrow 0$.
\end{proof}


It is routine to verify that $\widetilde{\Pi}$ is the L\'evy measure of a stable process with Laplace exponent 
$$\frac{\Gamma(1-\alpha)}{\alpha 2^{1+\alpha}\Gamma(2+\alpha)}\lambda^{1+\alpha}.$$
In particular, the L\'evy measure of the ${\tt Stable}(1+\alpha)$ process in Theorem \ref{thm:lim3} is $c\widetilde{\Pi}$ where 
$c=2\alpha(1+\alpha)/\Gamma(1-\alpha)$. 

We can also approximate the L\'evy measure directly from the convergence in Theorem \ref{thm:lim3}, where pre-limiting jump sizes are governed by the integer-valued table size
process, while limiting jump sizes are governed by the L\'evy measure. As in Proposition \ref{taudist}, we will denote by $\PP_1$ the distribution of the 
table size process starting from 1, and by $\zeta$ the hitting time of 0 under $\PP_1$.


\begin{cor} 
For every $\varepsilon>0$,  it is the case that as $n\rightarrow\infty$,
\begin{align*}
2\alpha n^{1+\alpha}\PP_{1}\left(\frac{\zeta}{2n} > \varepsilon \right) &\longrightarrow \Pi(\varepsilon,\infty)\\
\textrm{and}\quad\ \  \PP_{1}\left(\left.\frac{\zeta}{2n} \in\; \cdot\,  \right|\frac{\zeta}{2n} > \varepsilon  \right) &\longrightarrow \frac{\Pi\big(\,\cdot\,\cap(\varepsilon,\infty)\big)}{\Pi(\varepsilon,\infty)},
\end{align*}
where $\Pi=c\widetilde{\Pi}$ is as above and the second convergence is in the sense of weak convergence.
\end{cor}

\begin{proof}

In the setting of the proof of Theorem \ref{thm:lim3}, with $X_0\!=\!0$, set $X^{n}_{t}:=(2n)^{-1}X_{2n^{1+\alpha}t}$, and write $F_{X^{n}}$ and $F_{\overline{X}}$ for the 
L\'{e}vy measures of $X^n$ and $\overline{X}$ respectively. Noting that $X^{n}_{0} = 0$, a combination of Lemma \ref{Jacodcadlagconvergence} and Theorem \ref{thm:lim3} shows that 
$F_{X^{n}}(g)\rightarrow F_{\overline{X}}(g)$ for every $g\in C_{2}([0,\infty),\RR)$. 
Since $X$  is a compound Poisson process with drift, observe that it has L\'{e}vy measure $\alpha\PP_1(\zeta\in dx)$, so that 
$F_{X^{n}}(dx) = 2\alpha n^{1+\alpha} \PP_1\big((2n)^{-1}\zeta\in dx\big)$.
On the other hand, $F_{\overline{X}}$ is the L\'{e}vy measure of the limiting ${\tt Stable}(1+\alpha)$ L\'{e}vy process with Laplace exponent 
$\psi$. This L\'evy measure is $F_{\overline{X}}=\Pi$, as identified just above the statement of this corollary. 

Note that Lemma \ref{Jacodcadlagconvergence} yields vague convergence of the $\sigma$-finite measures $F_{X^{n}}$ to $F_{\overline{X}}$ on $(0,\infty)$.
Now, $F_{\overline{X}}$ is a measure with no atoms, and it is easy to show that this is equivalent to the claim that
$$
F_{X^{n}}\left(\varepsilon,\infty \right)\longrightarrow F_{\overline{X}} \left (\varepsilon,\infty \right),\qquad\mbox{for all $\varepsilon>0$}.
$$
Now, for any $x>\varepsilon$, observe that as $n\rightarrow\infty$,
$$
\PP_{1}\left(\left. \frac{\zeta}{2n} > x \right| \frac{\zeta}{2n} >  \varepsilon \right) = \frac{2\alpha n^{1+\alpha}\PP_{1}\left((2n)^{-1}\zeta > x\right)}{2\alpha n^{1+\alpha}\PP_{1}\big( (2n)^{-1}\zeta > \varepsilon\big)} \longrightarrow \frac{\Pi\left(x,\infty\right)}{\Pi\left(\varepsilon,\infty\right)},
$$
which entails the claimed  weak convergence of probability measures on $(\varepsilon,\infty)$. 
\end{proof}

\begin{rem}
\rm
It can be shown that $2\alpha n^{1+\alpha}\PP_{1}\big((n^{-1}Z_{2nt})_{t\geq 0} \in\, \cdot\,\big) \rightarrow \Theta$ vaguely on 
$\mathbb{D}([0,\infty),\RR)\setminus\{0\}$, where $\Theta$ is the ${\tt BESQ}(-2\alpha)$ excursion measure mentioned above Conjecture \ref{conj2}, with 
$\Theta(\zeta\in\,\cdot\,)=\Pi$. See \cite{fourauthor16,PartA,PartB} for the direct study of limiting structures consisting of ${\tt Stable}(1+\alpha)$ processes with 
${\tt BESQ}(-2\alpha)$ excursions in their jumps.

\end{rem}

\subsection*{Acknowledgements}

Dane Rogers was supported by EPSRC DPhil studentship award 1512540 and by a Merton doctoral completion bursary. We would like to thank Jim Pitman for pointing out some relevant references and Noah Forman, Soumik Pal and Douglas Rizzolo for allowing us to build on unpublished drafts from which several ideas here arose and that explored the special case $\alpha=1/2$ specifically. We thank Christina Goldschmidt and Lo\"ic Chaumont for valuable feedback on the thesis version of this paper, which also led to improvements here.   

%
%
%
%

\end{document}